\documentclass[a4paper,10pt]{article}

\usepackage[T1]{fontenc}
\usepackage[brazilian, english]{babel}
\usepackage{enumerate}
\usepackage{amsmath,amssymb,amsthm}
\usepackage{graphicx}
\usepackage{bbm}
\usepackage{dsfont}
\usepackage{cite}      
\usepackage{mathtools, nccmath, relsize}

\usepackage[hidelinks,bookmarksnumbered]{hyperref}
\hypersetup{pdfstartview=FitH}
\usepackage{tikz}
\usetikzlibrary{arrows.meta}
\usetikzlibrary{patterns}
\usetikzlibrary{shapes}
\usetikzlibrary{calc}
\usetikzlibrary{backgrounds}
\usetikzlibrary{decorations.pathreplacing}
\usetikzlibrary{decorations.pathmorphing}

  \newcounter{constant}

\def\arraypar#1{\parbox[c]{\textwidth - 2cm}{\centering #1}}

\usepackage{environ}
\NewEnviron{display}{
  \begin{equation}\begin{array}{c}
    \arraypar{\BODY}
  \end{array}\end{equation}
}

\newcommand{\I}{\mathds{1}}

\newcommand{\Ber}{\mathop{\mathrm{Ber}}\nolimits}
\newcommand{\Bin}{\mathop{\mathrm{Bin}}\nolimits}
\newcommand{\Cov}{\mathop{\mathrm{Cov}}\nolimits}
\newcommand{\Exp}{\mathop{\mathrm{Exp}}\nolimits}

\newcommand{\norm}[1]{\left\lVert#1\right\rVert}

\newcommand{\tS}{\smash{\tilde{S}}}
\newcommand{\tT}{\smash{\tilde{T}}}

\newcommand{\sB}{\smash{\mathsf{B}}}

\newcommand{\hH}{\smash{\hat{H}}}
\newcommand{\hP}{\smash{\hat{\PP}}}

\newcommand{\comp}{\mathsf{c}}

\newcommand{\cB}{\ensuremath{\mathcal{B}}}
\newcommand{\cC}{\ensuremath{\mathcal{C}}}

\newcommand{\cE}{\ensuremath{\mathcal{E}}}
\newcommand{\cF}{\ensuremath{\mathcal{F}}}

\newcommand{\cH}{\ensuremath{\mathcal{H}}}
\newcommand{\cI}{\ensuremath{\mathcal{I}}}

\newcommand{\cN}{\ensuremath{\mathcal{N}}}
\newcommand{\cO}{\ensuremath{\mathcal{O}}}
\newcommand{\cP}{\ensuremath{\mathcal{P}}}

\newcommand{\cR}{\ensuremath{\mathcal{R}}}

\newcommand{\cT}{\ensuremath{\mathcal{T}}}

\newcommand{\cX}{\ensuremath{\mathcal{X}}}

\newcommand{\EE}{\ensuremath{\mathbb{E}}}

\newcommand{\NN}{\ensuremath{\mathbb{N}}}
\newcommand{\PP}{\ensuremath{\mathbb{P}}}

\newcommand{\RR}{\ensuremath{\mathbb{R}}}
\newcommand{\ZZ}{\ensuremath{\mathbb{Z}}}

\theoremstyle{plain}
\newtheorem{teo}{Theorem}[section]

\newtheorem{prop}[teo]{Proposition}
\newtheorem{lema}[teo]{Lemma}
\newtheorem{coro}[teo]{Corollary}
\newtheorem{claim}[teo]{Claim}
\newtheorem{defi}[teo]{Definition}

\theoremstyle{remark}
\newtheorem{remark}{Remark}[section]

\title{Results on the contact process with dynamic edges or under renewals}

\author{Marcelo Hil\'ario\footnote{ICEx, Universidade Federal de Minas Gerais,
Belo Horizonte, Brazil.
E-mail: mhilario@mat.ufmg.com},
Daniel Ungaretti\footnote{IME, Universidade de S\~ao Paulo, S\~ao Paulo, Brazil.
E-mail: danielungaretti@gmail.com},
Daniel Valesin\footnote{Bernoulli Institute, University of Groningen,
Nijenborgh 9 9747 AG Groningen. Email: d.rodrigues.valesin@rug.nl},
Maria Eul\'alia Vares\footnote{Instituto de Matem\'atica, Universidade Federal
do Rio de Janeiro, RJ, Brazil. \!Email: eulalia@im.ufrj.br}}

\begin{document}
\maketitle

\begin{abstract}
We analyze variants of the contact process that are built by modifying the
percolative structure given by the graphical construction and develop
a robust renormalization argument for proving extinction in such models.
With this method, we obtain results on the phase diagram of two models:
the Contact Process on Dynamic Edges introduced by Linker and Remenik and a
generalization of the Renewal Contact Process introduced by Fontes, Marchetti,
Mountford and Vares.
\end{abstract}





\section{Introduction}
\label{sec:introduction}

The contact process was introduced by Harris~\cite{Har} as a Markov process
that models contact interactions on a lattice and has become one of the most
studied interacting particle systems ever since. It may alternatively be
defined in terms of a percolative structure usually referred to as a
\emph{graphical representation}~\cite{Har78}. In the classical setting, this is
done with the aid of infinitely many independent Poisson point processes (PPP)
which are suitably assigned to the sites and to the edges of the lattice. Using
the common interpretation of the contact process as a model for an infection,
the Poisson marks represent the space-time location where either a transmission
across an edge takes place or where a site is cured. Apart from providing an
appealing interpretation and allowing for a construction of the contact process,
the graphical representation is also an important tool in the study of the
process (and more generally, in other classes of interacting particle systems).
Notably, the proof by Bezuidenhout and Grimmett~\cite{BG} of the result that
the critical contact process dies out makes fundamental use of it.

By replacing the PPPs in the graphical representation by other types of point processes,
one is led to natural generalizations of the contact process. That is the case for the 
Contact Process on Dynamic Edges (CPDE)~\cite{LR} and the Renewal Contact Process (RCP)
\cite{FMMV,FGS,FMV,FMUV}. Even though the Markov property and other useful features like
the FKG inequality may no longer hold, the usual questions regarding survival or
extinction still make good sense and remain interesting from various aspects,
including the percolative perspective itself. 

In this paper, we call a \textit{Generalized 
Contact Process} (GCP) any process that is obtained from a percolative structure
of recovery and transmission marks in the same way as the contact process, but where the distribution
of these marks is given by some other point process.
Our contributions in the study of these processes are twofold. First, we develop
a robust renormalization approach that allows us to study the survival or
extinction for the GCP. Then we specialize to two  types of GCP
(the aforementioned CPDE and variants of the aforementioned RCP), and prove some results concerning their phase diagram.
We expect that our methods may be applicable in
greater generality, as for example, to allow the study of extinction for
contact processes in random environments with space-time correlations.
Next we provide a more detailed definition for the models to be considered and
present our main results.

\medskip
\noindent
\textbf{Generalized Contact Process.} Let $\cN_x$ and $\cN_{x,y}$ be point
processes on the line, indexed by the sites $x \in \ZZ^d$ and by the pairs of
nearest neighbors $x,y\in \ZZ^{d}$ respectively. Given an initial configuration
$\smash{\xi_0 \in \{0,1\}^{\ZZ^{d}}}$ we define a process $(\xi_t)_{t\geq 0}$
taking values in $\{0,1\}^{\mathbb{Z}^d}$ where $\xi_t(x) = 1$ (resp.
$\xi_t(x) =0$) is interpreted as $x$ being infected (resp.\ healthy) at time
$t$.  From the initial configuration the process evolves in time, with the
marks in $\cN_x$ and $\cN_{x,y}$ playing the roles of the instants of time when
$x$ may get cured, and when the infection may be transmitted across the edge
$xy$ respectively. More precisely, if $x, y \in \ZZ^{d}$ and $s < t$, we define
a path $\gamma$ from $(x,s)$ to $(y,t)$ as a c\`adl\`ag function $\gamma: [s,t]
\to \ZZ^{d}$ that fulfills the following properties:
\begin{itemize}
\item it does not contain any cure marks:
for every $u \in [s,t]$ $(\gamma(u), u) \notin \cN_{\gamma(u)}$;
\item its discontinuities have size one and only occur at transmission times:
    if $\gamma(u) \neq \gamma(u-)$ then $\gamma(u)$ and $\gamma(u-)$ are
    nearest neighbors in $\mathbb{Z}^d$ and $u \in \cN_{\gamma(u-), \gamma(u)}$.
\end{itemize}
The event that there is a path from $(x,s)$ to $(y,t)$ is denoted by
$(x,s) \rightsquigarrow (y,t)$. We define $\xi_t(y)$ as
\begin{equation*}
\text{$\xi_t(y) = 1$ if and only if we have
$(x,0) \rightsquigarrow (y,t)$ for some $x$ with $\xi_0(x)=1$}.
\end{equation*}
We may pick a set $A \subset \ZZ^{d}$ to be the set of initially infected
sites, by taking $\xi_0^{A} := \I_A$ in which case we write $\xi_t^{A}$ for
the resulting process. Identifying a configuration $\xi \in \{0,1\}^{\ZZ^{d}}$
with its set of infected sites $\{x \in \ZZ^{d};\; \xi(x) = 1\}$ allows us to
write $\xi^{A}_t = \cup_{x \in A} \xi^{\{x\}}_t$, a property known as
\textit{additivity}. For any starting configuration $A$, we define the
extinction time from $A$ as
\begin{equation*}
\tau^{A} := \inf\{t;\; \xi_t^{A} \equiv 0\}.
\end{equation*}
We say that the GCP dies out or that extinction occurs if $\tau^{\{0\}} <
\infty$, where $\{0\}$ represents the set containing only the origin.
On a vertex-transitive graph, additivity implies that if $\tau^{\{0\}} <\infty$
a.s.\ then the same holds for any finite starting set $A$.

In Section~\ref{sec:renormalization_scheme} we present a reformulation of the
renormalization approach used in~\cite{FMUV} that can be used for proving
extinction for the GCP under certain conditions.
This renormalization is applied to the CPDE and the RCP models. In both
models, $\cN_x$ and $\cN_{x,y}$ are independent renewal processes and to
emphasize this we denote them by $\cR_x$ and $\cR_{x,y}$. However, we remark
that the GCP includes more general point processes. By presenting the
construction in such generality we expect that it can be extended to other
settings like the contact processes on random environments containing
space-time correlations.

\medskip
\noindent
\textbf{Contact Process on Dynamic Edges.} Some of our results concern a model
of contact process on a dynamic random environment that was introduced
in~\cite{LR}.  

The environment is given by a dynamic percolation~\cite{Steif} on the edges of
the $\ZZ^d$ lattice. Initially each edge is independently declared open
with probability $p$ and closed otherwise. In the environment dynamics, each edge updates its
state to open or closed with respective rates $vp$ and $v(1-p)$. Hence the
environment is in equilibrium, given by the product of Bernoulli measures of
parameter $p$ and $v >0$ is the total rate at which edges update.
 
Conditional on the environment, the contact process evolves as following:
infected individuals try to transmit the infection to each  of its healthy
neighbors at a fixed rate $\lambda$ and heals at rate one. However, any
attempt for transmissions is only successful when it is allowed by the
environment, i.e.\ when it is done across an open edge. We write
$\mathbb{P}_{v,p,\lambda}$ for the joint law of the process and the environment
on some suitable probability space. More details about the model are
given in Section~\ref{sec:cpde}.

Monotonicity with respect to $\lambda$ allows us to define the critical
parameter
\begin{equation*}
\lambda_0(v,p)
    := \inf\{\lambda > 0;\;
        \PP_{v,p,\lambda} (\tau^{\{0\}} = \infty) > 0\}.
\end{equation*}
Let $p_c(d)$ be the critical point for Bernoulli bond percolation on $\ZZ^d$.
Our main result is the following:
\begin{teo}
\label{teo:extinction_cpde}
Consider the CPDE on $\ZZ^{d}$ with $d \ge 2$.
\begin{enumerate}[(i)]
\item For all $p < p_{c}(d)$ and $\lambda > 0$ there exists
    $v_0(p, \lambda, d)>0$ such that for any $v \in (0, v_0)$ the infection
    dies out almost surely.
\item For any~$p > p_c(d)$ we have
    \begin{equation}
    \label{eq_sup_vp}
    \sup\left\{\lambda_0(v,p'):v \ge 0,\; p' \in [p,1]\right\} < \infty.
    \end{equation}
\end{enumerate}
\end{teo} 

Several questions concerning the CPDE were investigated in~\cite{LR}, mostly
related to the behavior of the model as a function of the parameters $v$, $p$
or $\lambda$. The results obtained there hold for any infinite
vertex-transitive regular graph, the only exception being Theorem~2.4 which
corresponds to item \emph{(i)} in Theorem~\ref{teo:extinction_cpde} for the
graph $\ZZ$ (with the obvious adaptation that $p_c(1)=1$). The proof there relies on
arguments that are essentially $(1+1)$-dimensional and does not seem to extend
easily to higher spatial dimensions. However, the authors conjectured that the
result should still hold beyond $\ZZ$ and item \emph{(i)} answers this
conjecture affirmatively for $\ZZ^d$, $d \ge 2$.

If $\bar{\lambda}$ stands for the critical parameter of the contact process on
the static lattice then a simple coupling shows that $\lambda_0(v,p) \ge
\bar{\lambda}$.  However, as shown in Corollary~2.8 in \cite{LR} there are
choices for $v$ and $p$ for which $\lambda_0(v,p) = \infty$.  In fact their
results imply that if
\begin{equation*}
p_1 := \sup\{p > 0;\; \text{there is $v>0$ with
    $\lambda_0(v,p) = \infty$}\}
\end{equation*}
then $p_1 \in (0,1)$, see Figure~\ref{fig:CPDE_phase_diagram}.
Theorem~\ref{teo:extinction_cpde}(ii) implies that $p_1 \le p_c$.

\begin{figure}[h]
\centering
\begin{tikzpicture}[scale=2.5]
    \coordinate (p1) at (0,.4) node[left] at (p1) {$p_1$};
    \coordinate (pc) at (0,.6) node[left] at (pc) {$p_c$};
    \draw[->] (0,-.1) -- (0,1.2) node[left] {$p$};
    \draw[->] (-.2, 0) -- (3.1, 0) node[below] {$v$};
    \draw[dashed] (0,1) node[left] {$1$} -- (3,1) node[right] {};
    \draw[dashed] (0,1) -- ++(3,0) (p1) -- ++(3,0) (pc) -- ++(3,0);
    \fill[blue, opacity=.7] (p1) plot [smooth] coordinates
    {(p1) (1,.2) (3, .1)} -- (3,0) -- (0,0) -- cycle;
    \node at (.2, .2) {$\mathfrak{I}$};

    \draw[thick,decorate,decoration={brace,amplitude=5pt}]
        (3.2,1) -- (3.2, .6) node[midway, right=6pt, align=center]
        {$\lambda_0(\cdot, p)$ bounded \\ (Theorem~\ref{teo:extinction_cpde}(ii))};
    \draw[thick,decorate,decoration={brace,amplitude=5pt}]
        (3.2,.6) -- (3.2, 0) node[midway, right=6pt, align=center]
        {$\lambda_0(\cdot, p)$ unbounded \\
        (Theorem~\ref{teo:extinction_cpde}(i))};
\end{tikzpicture}
\caption{Phase diagram for the behavior of $\lambda_0(v,p)$. The immunity
    region $\mathfrak{I}$ in blue is the set of points with $\lambda_0(v,p) =
    \infty$. It falls below the line $p=p_1$, and the critical value $p_1$ is
    shown to be at most $p_c$ by Theorem~\ref{teo:extinction_cpde}(ii).  For
    CPDE on $\ZZ$ it holds $p_1 < p_c = 1$, but it is not known if this holds
    in general.}
\label{fig:CPDE_phase_diagram}
\end{figure}
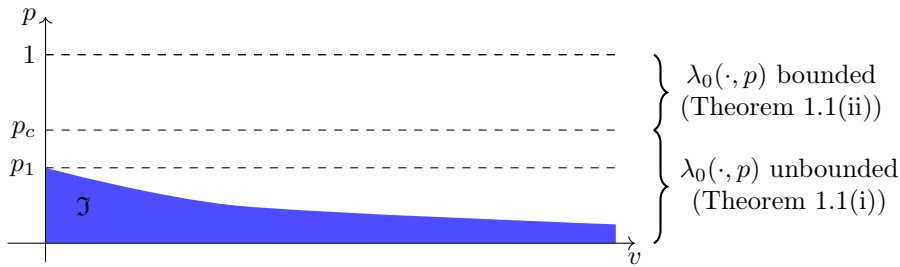

The proof for Theorem~\ref{teo:extinction_cpde}(ii) is based on a one-step
renormalization argument, while the proof for
Theorem~\ref{teo:extinction_cpde}(i) relies on the multiscale
renormalization construction presented in Section~2.

\medskip
\noindent
\textbf{Renewal Contact Process.}
The second model for which we apply the method in
Section~\ref{sec:renormalization_scheme} is a further generalization of the
Renewal Contact Process (RCP) that appeared in the series of papers~\cite{FMMV,
FMV, FMUV, FGS}.  There, the authors consider the GCP in which the
transmissions are governed by PPPs of rate $\lambda$ and the cures by renewal
processes whose interarrivals have a certain distribution $\mu$.  One can
define the critical parameter whose value depends on the specific choice of
distribution $\mu$ as:

\begin{equation*}
\lambda_c(\mu)
    := \inf\{\lambda > 0;\; \PP(\tau^{\{0\}} = \infty) > 0\}.
\end{equation*}
One of the goals for the investigation carried on in~\cite{FMMV, FMV, FMUV,
FGS}, was to relate the question whether $\lambda_c(\mu)$ is strictly positive
to the tail decay of $\mu$.

Analogously, if renewals with interarrival distribution $\nu$ were associated
to the transmissions while the cures were given by PPP's of rate $\delta$, then
one could define the critical parameter as
\begin{equation}
\label{eq:delta_c_old}
\delta_c(\nu)
    := \inf \{\delta > 0;\; \PP(\tau^{\{0\}} = \infty) = 0\}.
\end{equation}

The main novelty here is to allow for both the transmissions and the cures to
be determined by renewal processes. Their interarrival distributions will be
denoted $\mu$ and $\nu$, respectively. Let us denote $\PP_{\mu, \nu}$
the law of the resulting GCP on some suitable probability space.

A natural definition for the critical parameter is done by fixing one
interarrival distribution while scaling the other.  In fact, for any fixed
$\delta > 0$ let us define the distribution $\nu_{\delta}$ on $\RR_+$ given by
$\nu_{\delta}(t, \infty) := \nu(\delta t, \infty)$ for every $t\ge 0$.  Then,
we can define the critical parameter
\begin{equation}
\label{eq:delta_c_general}
\delta_c(\mu, \nu)
    := \inf \{\delta > 0;\; \PP_{\mu, \nu_{\delta}}(\tau^{\{0\}} = \infty) = 0\}.
\end{equation}
The definition in~\eqref{eq:delta_c_general} includes the one
in~\eqref{eq:delta_c_old} since $\delta_c(\mu, \Exp(1)) = \delta_c(\mu)$.
However, notice that monotonicity of $\PP_{\mu,
\nu_{\delta}}(\tau^{\{0\}} = \infty)$ in $\delta$ is not clear for general $\nu$.  A similar
definition, considering scalings for $\mu$, leads to the definition of
$\lambda_c(\mu, \nu)$.

We investigate the behavior of $\delta_c(\mu, \nu)$. To emphasize that we are
fixing the edge (or transmission) renewal processes and scaling the site (or
cure) processes, we adopt the more explicit name Edge Renewal Contact Process
(ERCP). Similarly, we write Site Renewal Contact Process (SRCP) in case we are
fixing the cure processes.

A first observation is that since we allow for transmissions that are given by general
renewal processes, it is not clear in principle whether the infection can spread to
infinitely many sites in finite time. In order to avoid such undesirable behavior, we only
study $\delta_c(\mu, \nu)$ for continuous $\mu$ and
$\nu$. Also, notice that if we suppress the cures in an ERCP, the
resulting model can be seen as a generalization of Richardson model,
which corresponds to the case $\mu$ is an exponential distribution.

Let us denote by $\cR$ the collection of renewal marks obtained from
some distribution $\mu$. Besides continuity, there are two hypotheses on
distributions that we use frequently.

The first one is a quantitative control on the renewal marks of a
heavy-tailed distribution: we say that $\mu$ satisfies condition~\eqref{eq:gap_t_epsilon}
if there exists $\epsilon_4 > 0$ and $t_0 > 0$ such that
\begin{equation}
\label{eq:gap_t_epsilon}
\PP(\cR \cap [t, t+t^{\epsilon_4}] \neq \emptyset)
    \le t^{-\epsilon_4} \quad \text{for $t \ge t_0$}.
    \tag{G}
\end{equation}
We interpret condition~\eqref{eq:gap_t_epsilon} as $\cR$ having
increasingly large gaps. It appeared already in~\cite{FMMV} and
in~\cite[Proposition~7]{FMMV} it is shown that~\eqref{eq:gap_t_epsilon}
holds whenever $\mu$ satisfies conditions A)-C) defined therein, which we
reproduce here:
\begin{enumerate}[A)]
\item There is $1 < M_1 < \infty$, $\epsilon_1>0$ and $t_1 > 0$ such that 
    \begin{equation*}
    \text{for every $t > t_1$,} \qquad
        \epsilon_1 \smash{\int_{[0,t]}} s \mu(\mathrm{d}s) < t \mu(t, M_1 t).
    \end{equation*}
\item There is $1 < M_2 < \infty$, $\epsilon_2>0$ and $r_2 > 0$ such that 
    \begin{equation*}
    \text{for every $r > r_2$,} \qquad
        \epsilon_2 \mu[M_2^{r}, M_2^{r+1}] \le \mu[M_2^{r+1}, M_2^{r+2}].
    \end{equation*}
\item There is $M_3 < \infty$, $\epsilon_3>0$ such that 
\begin{equation*}
    \text{for $t \ge M_3$,} \qquad
        t^{-(1-\epsilon_3)} \le \mu(t,\infty) \le t^{-\epsilon_3}.
    \end{equation*}
\end{enumerate}
Condition~C) controls the tail decay of $\mu$, while A) and B) concern its
regularity. Together, they provide a straightforward way to verify if $\mu$
satisfies~\eqref{eq:gap_t_epsilon}.
In particular, hypothesis~\eqref{eq:gap_t_epsilon} holds when
$\mu(t, \infty) = L(t)t^{-\alpha}$ with $\alpha \in (0,1)$, where~$L$ is a slowly 
varying function at infinity.

The second one is a moment condition that is behind a sufficiently fast decay
of correlations for events depending on $\cR$. In~\cite[Theorem~1.1]{FMUV} it
is proved that a sufficient condition for $\lambda_c(\mu) > 0$ in a SRCP($\mu$)
is the moment condition
\begin{equation}
\label{eq:moment_condition_rcp3}
\int_{1}^{\infty} x
    \exp\Bigl[ \theta (\ln x)^{1/2} \Bigr] \mu(\mathrm{d}x)
    < \infty
\quad \text{for some $\theta > \sqrt{(8 \ln 2) d}$}.
    \tag{M}
\end{equation}
Condition~\eqref{eq:moment_condition_rcp3} goes in the opposite
direction of~\eqref{eq:gap_t_epsilon}; it implies that it is hard to find
large intervals without renewal marks, see Lemma~\ref{lema:unif_estimates}(i).
Moreover, it is slightly stronger than finite first moment: for instance, if
$\mu$ has a finite $(1+\varepsilon)$-moment then it
satisfies~\eqref{eq:moment_condition_rcp3}.

Now, we are ready to state our results about ERCP. Define
\begin{equation*}
r_t := \max \{\norm{x}_1;\; (0,0) \rightsquigarrow (x,t)\
    \text{in ERCP without cures}\}.
\end{equation*}

We estimate the speed of infection of an ERCP without cures. This is based
on a comparison with a toy model of iterated percolation, see
Section~\ref{sub:growth_ercp}.
\begin{teo}
\label{teo:growth_ercp}
Let $d \ge 1$ and $\mu$ be a continuous distribution.
\begin{enumerate}[(i)]
\item For any $a > 1$ it holds
    $\varlimsup\limits_{t \to \infty} \dfrac{r_t}{t (\ln t)^{a}} = 0$, almost surely.

\item Suppose $\mu$ satisfies~\eqref{eq:gap_t_epsilon}. Then, the process $r_t$ has
    sublinear growth: for every $\rho \in (0, \epsilon_4)$ we have
\begin{equation}
    \label{eq:sublinear_ERCP_rcp1}
    \varlimsup_{t \to \infty} \frac{r_t}{t^{1-\rho}} = 0
    \qquad \text{almost surely}.
\end{equation}
\end{enumerate}
\end{teo}

Our main contribution in the investigation of ERCP is a set of conditions on
$\mu$ and $\nu$ under which we can show whether $\delta_c(\mu,\nu)$ is trivial
or not.  Heuristically, when $\mu$ is a heavy-tailed distribution,
Theorem~\ref{teo:growth_ercp} shows that the speed of the infection in an
environment without cures is slow and any rate $\delta > 0$ of cure is
sufficient for it to die out. On the other hand, when $\mu$ and $\nu$ have a fast
tail decay one expects a non-trivial phase transition, similar to what is
observed in the standard Contact Process.
\begin{teo}
\label{teo:ercp_extinction}
Let $\mu, \nu$ be continuous interarrival distributions and consider a
ERCP($\mu, \nu$) in $\ZZ^{d}$.
\begin{enumerate}[(i)]
    \item If $d\ge 1$, $\mu$ satisfies~A)-C) and
    $\nu$ satisfies $\int x^{n} \nu(\mathrm{d}x) < \infty$ for all $n\ge 1$,
    then the ERCP($\mu, \nu_{\delta}$) dies out almost surely,
    for each $\delta>0$, i.e.\ , $\delta_c(\mu,\nu) = 0$.
\item If $d \ge 2$ and $\mu$ has finite first moment then $\delta_c(\mu, \nu) > 0$.
\item If $d \ge 1$ and $\mu$ and $\nu$ satisfy~\eqref{eq:moment_condition_rcp3},
    then $\PP_{\mu, \nu_{\delta}}(\tau^{\{0\}} = \infty) = 0$ for sufficiently large
    $\delta$, i.e.\ , $\delta_c(\mu,\nu) <\infty$.
\end{enumerate}
\end{teo}

\begin{remark}
The restriction to $d\ge 2$ in (ii) allows a simpler argument by using at most
once the transmission process at each given edge, therefore avoiding the
dependencies between various residual times. A suitable extension to $d=1$ is
expected.
\end{remark}

\noindent
\textbf{Related works.}
Let us now comment on some related works, apart from~\cite{FMMV,FGS,FMV,FMUV}
and~\cite{LR}, which were already mentioned.  There is a substantial
literature on the contact process on static random environments, that is,
versions of the contact process in which the recovery and transmission rates
may vary spatially, and are sampled from some environment distribution, but
the dynamics is still driven by Poisson point processes. Due to this last
point, these models are fundamentally different from the ones we consider, but
there are similarities in the line of investigation and the analysis.
Klein~\cite{Klein} considers an environment obtained  from  i.i.d. recovery
and transmission rates, and gives a condition on the environment distribution
that guarantees almost sure extinction; his method is a multi-scale
construction that has similarities to the one we employ. Newman and
Volchan~\cite{NV} consider a one-dimensional recovery environment (and
transmissions with constant rate~$\lambda$), and give a condition on the
environment distribution that guarantees survival regardless of~$\lambda$. See
also~\cite{Andjel, BDS, GM, Lig92}.

\medskip
\noindent
\textbf{Summary of the paper.}
This paper is organized as follows. In Section~\ref{sec:renormalization_scheme}
we develop the renormalization argument for a GCP. Section~\ref{sec:cpde}
contains the results for CPDE and Section~\ref{sec:ercp} contains the results
for ERCP.

\medskip
\noindent
\textbf{Acknowledgements.}
The research of MH was partially supported by CNPq grants `Projeto Universal'
(406659/2016-8) and `Produtividade em Pesquisa' (312227/2020-5) and by FAPEMIG
grant `Projeto Universal' (APQ-02971-17).
DU was supported by grant 2020/05555-4, S\~ao Paulo Research Foundation (FAPESP).
MEV was partially supported by CNPq grant 305075/2016-0 and FAPERJ CNE grant
E-26/202.636/2019.

\section{Renormalization Scheme for Generalized Contact Process}
\label{sec:renormalization_scheme}



As mentioned in the introduction, in order to prove extinction for GCP we
develop a version of the renormalization in~\cite{FMUV}:
\begin{itemize}
\item The previous construction was developed in the context of Renewal Contact
    Process. We highlight that it can be actually applied to the GCP in general and
    instead of considering a sequence of boxes $B_n = [0,2^{n}]^d \times [0, h_n]$
    we allow spatial dimensions to grow faster, which can be useful
    to decouple variations of the Contact Process with space correlations.

\item As already used in \cite{FMUV}, crossing events on different scales are
    related via the definition of a single event we call a
    \textit{half-crossing}. We now introduce the notion of a
    \textit{hierarchy} of boxes, reminiscent of the arguments
    from~\cite{Rath,Sznitman}. Whenever one has a half-crossing of a
    large scale box, the hierarchy encodes the structure of smaller scale
    boxes that are also half-crossed. This provides an alternative way of
    using this renormalization to prove extinction of the infection, and we
    apply it to CPDE in Section~\ref{sub:teo_cpde_i}.
\end{itemize}

The choice of scales depends on some parameters that need to be tuned in order
for the argument to work. This tuning depends on the specific point processes
we choose for the model.

\subsection{Main events}
\label{sub:main_events}

We begin recalling the definition of a general space-time crossing.
\begin{defi}[Crossing]
Given space-time regions $C, D, H \subset \ZZ^{d} \times \RR$ we say there
is a crossing from $C$ to $D$ in $H$ if there is a path
$\gamma:[s, t] \to \ZZ^{d}$ such that $\gamma(s) \in C$, $\gamma(t) \in D$
and for every $u \in [s,t]$ we have $(\gamma(u),u) \in H$.
\end{defi}

Given $a = (a_1, \ldots, a_d)$, and $b = (b_1, \ldots, b_d)$ with $a_i < b_i$
for every $i$, let $[a, b] = \prod_{i=1}^{d} [a_i, b_i]$ and consider the
space-time box $B := [a,b] \times [s, t]$ whose projection into the spatial
coordinates is $[a, b]$. For each $1 \le j \le d$ we denote by
\begin{equation*}
    \partial_{j}^{-}B := \{(x, u) \in B;\; x_j = a_j\}
    \quad \text{and} \quad
    \partial_{j}^{+}B := \{(x, u) \in B;\; x_j = b_j\}
\end{equation*}
the face of $B$ that is perpendicular to direction $j$.
The hyperplane $\{(x, u) \in \ZZ^{d} \times \RR;\; x_j=\frac{a_j+b_j}{2}\}$
divides $B$ into two half-boxes, $B_j^{-}$ and $B_j^{+}$, that contain faces
$\partial_{j}^{-}B$ and $\partial_{j}^{+}B$, respectively.
Using this notation, four crossing events of the box
$B = [a,b] \times [s,t]$ will be important in our investigation.
\begin{description}
\item[\textbf{Temporal crossing.}]
Event $T(B)$ in which there is a path from $[a,b]\times \{s\}$ to
$[a,b]\times \{t\}$ in $B$.

\item[\textbf{Temporal half-crossing.}]
Event $\tT(B) := T([a,b] \times [s, \frac{t+s}{2}])$.
In words, we have a temporal crossing from the bottom of $B$ to
the middle of its time interval.

\item[\textbf{Spatial crossing.}]
For some fixed direction $j \in \{1,\ldots, d\}$ we $S_j(B)$ as the event that
there is a crossing from $\partial_{j}^{-}B$ to $\partial_{j}^+B$ in
$B$, i.e., there is a crossing connecting the opposite faces of $B$
that are perpendicular to direction $j$.

\item[\textbf{Spatial half-crossing.}]
For some fixed direction $j \in \{1,\ldots, d\}$, we define events
$\tS_{j,+}(B) := S_j(B_j^{+})$ and $\tS_{j,-}(B) := S_j(B_j^{-})$, in which
we have a spatial crossing in $B$ of a half-box connecting the opposite faces of
direction $j$. To ease notation, we write $\tS_{j,+} = \tS_j$ and
$\tS_{j,-} = \tS_{j+d}$, allowing indices $1 \le j \le 2d$.
\end{description}

Given a box $B$, consider the event
\begin{equation}
\label{e:def_H(B)}
H(B) := \tT(B)\cup \mathsmaller{\bigcup\limits_{j=1}^{2d} \tS_j(B)}
\end{equation}
that we refer to as \textit{half-crossing} of $B$; this event will play a
central role in the renormalization approach to be developed in the
next sections.  Our first aim is to show that $H(B)$ satisfies the so-called
cascading property, meaning that its  occurrence implies the existence of two
well-positioned smaller boxes inside $B$ which are also half-crossed.

\subsection{Cascading property for half-crossing events; hierarchies}
\label{sub:cascading_half_crossings}

In what follows we will analyze half-crossing events inside boxes of type $B_k
:= [-l_k, l_k]^{d} \times [0, h_k]$ where $(l_k) \subset \NN$ and $(h_k)
\subset \RR$ are increasing sequences to be determined later. They must be
interpreted as sequences of spatial and temporal scales along which we analyze
occurrence of half-crossing events. In fact, if the origin starts infected,
i.e.\ if $\xi_0(\mathbf{0}) >0$ and the resulting infection from that point survives
till time $h_k$ then either $T(B_k)$ occurs or the infection must leave box
$B_k$ through some of its faces $\partial^+_{j}B_k$ or $\partial_{j}^{-}B_k$
for $1 \le j \le d$. Thus, one can write

\begin{equation}
\label{eq:survival_prob}
\PP(\tau^{\{0\}} = \infty)
    \le \PP\Big(\tT(B_k) \cup \mathsmaller{\bigcup\limits_{j=1}^{2d} \tS_j(B_k)}\Big)
    = \PP(H(B_k)).
\end{equation}

Any space-time translation of the box $B_k$ is called a \textit{scale-$k$} box.
This section is devoted to proving a deterministic lemma that relates
half-crossings of boxes at two successive scales. In
Lemma~\ref{lema:cascading_half_cross} we prove that for this sequence of boxes
the event $H(B_k)$ is \textit{cascading}, meaning that its occurrence implies
the occurrences of two similar events inside disjoint boxes from the previous
scale. Moreover, we are able to

\begin{itemize}
\item find an upper bound (uniform in $k$) for the amount of pairs of boxes
    that we need to look at in order to find these two half-crossings;

\item control the positions of such pairs of boxes, obtaining that they might
    be taken well-separated in space and time. For a class of examples,
    this allows to decouple the corresponding half-crossing events.
\end{itemize}

This will ultimately allow us to control the right-hand side
in~\eqref{eq:survival_prob} which is useful for proving existence of regimes
when the infection dies out.

Let us now describe the rate of growth for the sequences of scales $(l_k)$ and
$(h_k)$ to be considered.  Given the initial scales $l_0 \in \mathbb{N}$ and
$h_0 \in \mathbb{N} $ and two constants $\alpha, \beta \in \NN$ we define
recursively
\begin{equation}
l_{k+1} = \alpha l_k
\quad \text{and} \quad
h_{k+1} = \beta h_k,
\qquad \text{for $k \ge 0$}.
\label{e:scales}
\end{equation}
Note that $l_k$ and $h_k$ grow exponentially fast. For other possibilities of
scale progression, see Remark~\ref{r:faster_than_exp_growth}.

We are now ready to state and prove the main results in this section. Temporal
and spatial half-crossing will be treated separately.  Let us begin with the
temporal ones.

\begin{lema}[Temporal half-crossings]
\label{lema:temporal_half_cross}
Fix $n \ge 1$ and $\beta \ge 6$.
There are collections $\cB_0 = \cB_{0,n}$ and $\cB'_0 = \cB'_{0,n}$ of
scale-$(n-1)$ boxes such that
\begin{equation*}
\tT(B_n) \subset \bigcup_{(B,B') \in \cB_0 \times \cB_0'} H(B) \cap H(B').
\end{equation*}
Moreover, we may assume that $\cB_0$ and $\cB_0'$ have $(2\alpha-1)^{d}$
elements each, and that the vertical distance between any pair of boxes
$B \in \cB_0$ and $B' \in \cB'_0$ is ${(\beta/2 - 2)h_{n-1}}$.
\end{lema}

\begin{proof}
By construction we have $h_n = \beta h_{n-1}$.
The event $\tT(B_n)$ entails the two following temporal crossings
\begin{equation*}
    T([-l_n,l_n]^{d} \times [0,h_{n-1}])
    \quad \text{and} \quad
    T([-l_n,l_n]^{d} \times [(\beta/2 - 1) h_{n-1}, (\beta/2) h_{n-1}]).
\end{equation*}
For $z \in [-\alpha, \alpha - 1] \cap \ZZ =: Z_{\alpha}$ let us define
\begin{equation*}
    I_z := l_{n-1}z + [0,l_{n-1}]
\end{equation*}
that forms a covering of $[-l_n,l_n]$ by $2\alpha$ intervals of length
$l_{n-1}$. On $T([-l_n,l_n]^{d} \times [0,h_{n-1}])$ we can find a path
$\gamma: [0,h_{n-1}] \to [-l_n,l_n]^{d}$ spanning the box in the temporal
direction. Let us consider its range $\cI = \gamma([0,h_{n-1}])$.
Projecting $\cI$ into each one of the coordinate directions $j$ yields
discrete intervals $\cI_j \subset [-l_n,l_n]$. Define the \textit{box
count} of $\cI_j$ as
\begin{equation}
\label{eq:projection_box_count}
c_j := \min\{|I|;\;
    I \subset Z_{\alpha},\ \cI_j \subset \cup_{z\in I} I_z\}.
\end{equation}
We decompose $T([-l_n,l_n]^{d} \times [0,h_{n-1}])$ according to the
values assumed by each $c_j$.

If for every $1\le j \le d$ we have $c_j \le 2$ then the whole path
$\gamma$ is contained inside a $d$-dimensional box with side length
$2l_{n-1}$. In this case, we can choose some
$z \in (Z_{\alpha} \setminus \{\alpha - 1\})^{d}$ such that
\begin{equation*}
    \cI \subset l_{n-1} z + [0,2l_{n-1}]^{d},
\end{equation*}
and the number of possible $z$ is given by $(2\alpha-1)^{d}$.

Now, let us consider the case in which some $c_j \geq 3$ and thus
$\cI$ is not contained in some of the boxes with side length $2l_{n-1}$
described above. In this case, we refine the argument by considering time. For
any time $t \in [0,h_{n-1}]$ we define $\cI(t) := \gamma([0,t])$ and for any
fixed direction $j$ we consider its projection $\cI_j(t)$ and its
box count $c_j(t)$. Define
\begin{equation*}
t_1 := \inf\{t \in [0,h_{n-1}];\; \exists 1 \le j \le d
        \text{ such that $c_j(t) \geq 3$}\}.
\end{equation*}
Since $\gamma$ can only change value when there is transmission to a neighboring
site, at time $t_1$ we have $c_{j_0}(t_1-) = 2$ and $c_{j_0}(t_1) = 3$ for some
special direction $j_0$ and $c_j(t_1) \le 2$ for every other direction. Thus, there
is $z \in (Z_{\alpha} \setminus \{\alpha - 1\})^{d}$ such that
\begin{equation*}
    \cI(t_1-) \subset l_{n-1} z + [0,2l_{n-1}]^{d}
    \quad \text{but} \quad
    \cI_{j_0}(t_1) \nsubseteq l_{n-1} z + [0,2l_{n-1}]^{d}
    \quad \text{and} \quad
    c_{j_0}(t_1) = 3.
\end{equation*}
Notice that this means path $\gamma$ must have crossed a half-box of
$l_{n-1} z + [0,2l_{n-1}]^{d}$ on direction $j_0$ during time
interval $[0,t_1] \subset [0, h_{n-1}]$, see Figure~\ref{fig:temp_project_gen}.
In any case, we have that $H(B)$ happens for some box $B$ in
\begin{equation*}
\cB_0
    := \bigl\{(l_{n-1} z + [0,2l_{n-1}]^{d}) \times [0, h_{n-1}];\;
        z \in (Z_{\alpha} \setminus \{\alpha - 1\})^{d}\bigr\}.
\end{equation*}
Applying the same argument for event
$T([-l_n,l_n]^{d} \times [(\beta/2 - 1) h_{n-1}, (\beta/2) h_{n-1}])$,
we conclude that we can take $\cB_0'$ as the vertical translation of boxes
of $\cB_0$ by $(\beta/2 - 2) h_{n-1}$.
\end{proof}

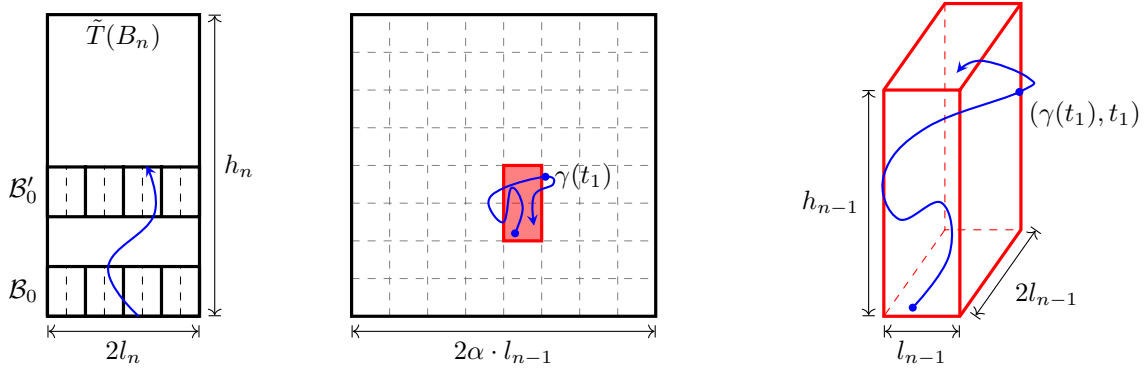
\begin{figure}
\centering
\begin{tikzpicture}[scale=1,
    dot/.style={fill, minimum size=3pt, outer sep=0pt,
                    inner sep=0pt, circle, blue},
    z={(55:.7cm)}
    ]
\begin{scope}[shift={(-1,0)}]
    \draw[very thick] (-3,0) rectangle (-1,4);
    \node[below] at (-2, 4) {$\tT(B_n)$};
    \node[left] at (-3,.33) {$\cB_0$};
    \draw[xstep=.25, ystep=.66, dashed]
    (-3,0) grid (-1,.66);
    \draw[xstep=.5, ystep=.66, very thick]
    (-3,0) grid (-1,.66);
    \node[left] at (-3,1.67) {$\cB'_0$};
    \draw[xstep=.25, ystep=.66, dashed]
    (-3,1.31) grid (-1,2);
    \draw[xstep=.5, ystep=.66, very thick]
    (-3,1.31) grid (-1,2);
    \draw[thick, tension=.7, blue, ->, >=stealth] plot [smooth] coordinates
    {(-1.8,0) (-2.2,.66) (-1.6, 1.34) (-1.7, 2)};
\draw[|<->|] (-3,-.2) -- (-1,-.2)
    node[midway, below] {$2l_{n}$};
\draw[|<->|] (-.8,0) -- (-.8,4) node[midway, right] {$h_{n}$};
\end{scope}

\draw[step=.5, dashed, help lines] (0,0) grid (4,4);
\draw[very thick] (0,0) rectangle (4,4);
\draw[|<->|] (0,-.2) -- (4,-.2)
    node[midway, below] {$2\alpha \cdot l_{n-1}$};
\draw[very thick, red] (2,1) rectangle (2.5,2);
\fill[opacity=.5, red] (2,1) rectangle (2.5,2);
\coordinate[dot] (A) at (2.15, 1.1);
\coordinate (B) at (2.4, 1.2);
\coordinate[dot] (t1) at (2.55, 1.85);
\node[right] at (t1) {$\gamma(t_1)$};
\draw[thick, tension=.7, blue, ->, >=stealth] plot [smooth] coordinates
    {(A) (2.25, 1.45) (2.1, 1.7) (2, 1.25) (1.8, 1.5) (2, 1.7)
        (t1) (2.65, 1.75) (2.4, 1.6) (B)};

\draw[very thick, red] (7,0)
    -- ++( 1,0,0) -- ++(0,0, 2) -- ++(0,3,0)
    -- ++(-1,0,0) -- ++(0,0,-2) -- cycle; 
\draw[red, dashed] (7,3,2) -- (7,0,2) -- (8,0,2) (7,0,2) -- (7,0,0);
\draw[very thick, red] (8,0,0) -- (8,3,0) -- (7,3,0) (8,3,0) -- (8,3,2);
\draw[|<->|] (6.8,0) -- (6.8,3) node[midway, left] {$h_{n-1}$};
\draw[|<->|] (7,-.2) -- (8,-.2) node[midway, below] {$l_{n-1}$};
\draw[|<->|] (8.2,0,0) -- (8.2,0,2) node[midway, below right] {$2l_{n-1}$};
\coordinate[dot] (3dA)  at (7.3, 0.0, 0.2);
\coordinate      (3dB)  at (7.8, 3.0, 0.4);
\coordinate[dot] (3dt1) at (8.1, 2.0, 1.7);

\node[below right] at (3dt1) {$(\gamma(t_1), t_1)$};
\draw[thick, tension=.7, blue, ->, >=stealth] plot [smooth] coordinates {
    (3dA) (7.5, 0.3, 0.9) (7.2, 0.7, 1.4) (7.0, 1.0, 0.5) (6.6, 1.3, 1.0)
    (7.0, 1.7, 1.4) (8.1, 2.0, 1.7) (8.3, 2.3, 1.5) (7.8, 2.7, 1.2) (3dB)};
\end{tikzpicture}
\caption{An illustration for the argument in Lemma~\ref{lema:temporal_half_cross} when $d=2$.
On the event $\tT(B_n)$, $H(B) \cap H(B')$ occurs for a box $B \in
\cB_0$ and another $B' \in \cB'_0$.
For $\cB_0$, when the projection of the
temporal crossing into space coordinates is not contained in one of the $(2\alpha-1)^{d}$ sub-boxes of side length $2l_{n-1}$ a spatial crossing of a half-box of scale
$n-1$ must occur.}
\label{fig:temp_project_gen}
\end{figure}

We now turn our attention to spatial half-crossings for which a similar result also holds.
\begin{lema}[Spatial half-crossing]
\label{lema:spatial_half_cross}
Assume that $\alpha \ge 4$.
Let $n \ge 1$ and $1 \le j \le 2d$.
There are collections $\cB_j = \cB_{j,n}$ and $\cB'_j = \cB'_{j,n}$ of scale-$(n-1)$
boxes such that
\begin{equation*}
\tS_j(B_n)
    \subset \bigcup_{(B,B') \in \cB_j \times \cB_j'} H(B) \cap H(B').
\end{equation*}
Moreover, we may assume that $\cB_j$ and $\cB_j'$ have $(2\beta-1) \cdot (2\alpha-1)^{d-1}$
elements and that any pair of boxes $B \in \cB_j$ and $B' \in \cB_j'$ have spatial distance at
least $(\alpha/2-2) 2l_{n-1}$.
\end{lema}

\begin{proof}
By symmetry, we can assume $j=1$.
On the event $S_1([0,l_n] \times [-l_n,l_n]^{d-1} \times [0,h_n])$
we have a half-crossing of box $B_n$, that entails the crossing of two
smaller boxes:
\begin{equation*}
S_1([0, 2l_{n-1}] \times [-l_n,l_n]^{d-1} \times [0,h_n])
\quad \text{and} \quad
S_1([l_n-2l_{n-1},l_n] \times [-l_n,l_n]^{d-1} \times [0,h_n]).
\end{equation*}
Similarly to what was done in the proof of
Lemma~\ref{lema:temporal_half_cross}, we will build a collection $\cB_1$
inside the first box and take $\cB_1'$ as a translation of $\cB_1$.  Hence
the spatial distance of boxes in $\cB_1$ and $\cB_1'$ is at least
$(\alpha/2-2) 2l_{n-1}$. Consider the collection of boxes
\begin{equation*}
\cC
    := \bigl\{(l_{n-1} z + [0,2l_{n-1}]^{d}) \times [0,h_n];\;
        z \in \{0\} \times (Z_\alpha \setminus \{\alpha-1\})^{d-1}\bigr\}.
\end{equation*}

Consider a path $\gamma: [s_1,t_1] \to \ZZ^{d}$ that realizes
$S_1([0, 2l_{n-1}] \times [-l_n,l_n]^{d-1} \times [0,h_n])$ and let
$\cI_j$ be the projection of $\gamma([s_1,t_1])$ on direction $j$ and $c_j$
be its box count, i.e.,
\begin{equation*}
c_j := \min\{|I|;\;
    I \subset Z_{\alpha},\ \cI_j \subset \cup_{z\in I} I_z\}.
\end{equation*}
Like in the previous lemma, if $c_j \le 2$  for every $2 \le j \le d$,
we can ensure that $\gamma$ is contained in some box $B \in \cC$ and
$S_1(B)$ happens. On the other hand, if some $c_j \geq 3$ then $\tS_j(B)$
happens for some box $B \in \cC$.
In both cases, the crossing of our smaller box implies the occurrence
of some half-crossing of a box $B \in \cC$ inside it, of the form
$[0,2l_{n-1}]^{d} \times [0,h_n]$. Finally, we adjust the time dimension of $B$
with a similar argument. Denote by $\pi(B)$ the space-projection of a
space-time box $B$. Define
\begin{equation*}
\cB_1 :=
    \{
        \pi(B) \times [i h_{n-1}, (i+1)h_{n-1}];\;
        B \in \cC, 0 \le i \le \beta-1, i \in 1/2 + \ZZ
    \}.
\end{equation*}
Our path $\gamma$ ensures that either we have $\tS_{j}(B)$ for some
direction $1 \le j \le d$ and box $B \in \cB_1$ or we have
$\tT(B)$ for some $B \in \cB_1$. It is easy to check that $\cB_1$ has
$(2\beta-1)\cdot (2\alpha-1)^{d-1}$ elements.
\end{proof}

Putting together Lemmas~\ref{lema:temporal_half_cross}
and~\ref{lema:spatial_half_cross} we readily obtain a result that relates the occurrence of half crossings at successive scales:
\begin{lema}[Cascading half-crossings]
\label{lema:cascading_half_cross}
For any $n \ge 1$, $\alpha \ge 4$ and $\beta \ge 6$ it holds
\begin{equation}
\label{e:entropy}
H(B_n) \subset
    \bigcup_{j=0}^{2d}
    \Bigl(
    \bigcup_{(B,B') \in \cB_j \times \cB_j'} H(B) \cap H(B')
    \Bigr).
\end{equation}
Hence, whenever we have a half-crossing of a scale-$n$ box we can find
half-crossings of two scale-$(n-1)$ boxes $B$ and $B'$ that either have
vertical distance at least $(\beta/2 - 2)h_{n-1}$ or have spatial distance at
least $(\alpha/2-2) 2l_{n-1}$.
Moreover, we can find such pair considering at most
\begin{equation*}
C(d, \alpha, \beta)
    := \bigl((2\alpha-1)^{d}\bigr)^2 +
        (2d)\cdot \bigl((2\beta-1)\cdot (2\alpha-1)^{d-1}\bigr)^2
\end{equation*}
pairs of boxes $(B,B')$.
\end{lema}

Lemma~\ref{lema:cascading_half_cross} provides explicit control on the amount
and on the position of the boxes where the half-crossings are found when moving
from one scale to the previous one. This allows to derive upper bounds
for the probability of half-crossings at large scales as we explain next. We
will present two possible approaches for obtaining such upper bounds. One of
them is to derive a contracting inequality relating the probability of the
crossing events at two successive scales. The other one is to move all the way
down to the bottom scale obtaining what we call an hierarchical structure.

\medskip
\noindent
\textbf{Recurrence inequality.}
Let $u_n := \sup_{(x,t)} P(H((x, t) + B_n))$.
Lemma~\ref{lema:cascading_half_cross} implies
\begin{equation}
\label{e:recurrence_ineq}
u_n
    \le C(\alpha, \beta, d) u_{n-1}^{2} +
        C(\alpha, \beta, d) \max_{(B,B')} \Cov(\I_{H(B)}, \I_{H(B')}),
\end{equation}
where the maximum runs over pairs of scale-$(n-1)$ boxes $(B,B')$ ranges in $ \cB_j \times \cB_j'$ with $0 \le j \le 2d$ and $C(\alpha, \beta, d)$ is given by \eqref{e:entropy}.
Provided that one is able to obtain good upper bounds on the covariance
of the events $H(B)$ and $H(B')$, then \eqref{e:recurrence_ineq} becomes a contraction, and therefore $u_n \downarrow 0$.
This is a very common strategy in renormalization.

\begin{remark}
\label{r:faster_than_exp_growth}
Notice that although we have worked with sequences $l_k, h_k$ that
grow at an exponential rate, Lemma~\ref{lema:cascading_half_cross} allows us to consider more general scale progressions.
In fact we may consider the more general relations
\begin{equation*}
l_{k+1} = \alpha_k l_k
\quad \text{and} \quad
h_{k+1} = \beta_k h_k,
\qquad \text{for $k \ge 1$},
\end{equation*}
and $H(B_n)$ is contained in a union of at most $C(d, \alpha_{n-1},
\beta_{n-1})$ pairs of scale-$(n-1)$ boxes.
For instance, one can use
sequences that grow faster than exponentially, in order to obtain better
decoupling inequalities, i.e.\ better bounds on the covariance.
This approach has been applied in~\cite[Theorem~1.1]{FMUV}.
One difference there is that it does not focus on events $H(B_n)$ but instead consider spatial and temporal crossings separately.
In the construction therein $\alpha = 2$, which means that pairs of boxes $(B,B')$ with $B = B'$ are considered.
Here, we have opted to focus on $\alpha \ge 4$ to ensure that the boxes considered are better separated.
\end{remark}

\medskip
\noindent
\textbf{Hierarchy.} Lemma~\ref{lema:cascading_half_cross} also allows for the
definition of a \textit{hierarchy} of boxes.  The idea is simple: fix $k\ge 1$
and consider box $B_k = [-l_k,l_k]^{d} \times [0, h_k]$.  Using
Lemma~\ref{lema:cascading_half_cross}, the original scale-$k$ box gives birth
to a pair of scale-$(k-1)$ boxes. After iterating the use of the same lemma we
end up with a collection of $2^{k}$ scale-$0$ boxes, all of which have been
half-crossed. This structure of boxes can be encoded via a binary tree,
in a construction that is similar to the one in~\cite{Rath} regarding random
interlacements.

Let us introduce some notation. We use words $a \in \{0,1\}^{n}$ to encode the
leaves of a binary tree of depth $n$. Consider that $\varnothing$ is the root
vertex, and $0$ and $1$ denote the left and right children of $\varnothing$,
respectively. We append digits to the right of a word $a \in \{0,1\}^n$ in
order to create longer words, e.g., $a1 \in \{0,1\}^{n+1}$ is the word that
encodes the right child of $a$. Morever, for $a,b \in \{0,1\}^{n}$ we define
the depth of $a$ by $|a| := n$ and define $a \wedge b \in \cup_{i=0}^{n}
\{0,1\}^{i}$ to be the common ancestor of $a$ and $b$ of highest depth. In
other words, $a$ and $b$ are descendants of $a \wedge b$ but there is no $c$
child of $a \wedge b$ that has both $a$ and $b$ as descendants.

For a fixed $k \ge 1$ we consider box $B_k$ as the root of the hierarchy,
and encode its descendants obtained via Lemma~\ref{lema:cascading_half_cross}
by a binary tree. More precisely, a collection of boxes
\begin{equation*}
    \cH_k = \{B_k(a);\; a \in \cup_{i=0}^{k} \{0,1\}^{i}\}
\end{equation*}
is called a \textit{hierarchy} of $B_k$ if $B_k(\varnothing) = B_k$, and for
every $a \in \cup_{i=0}^{k-1} \{0,1\}^{i}$ the boxes $B_k(a0)$ and $B_k(a1)$ are
disjoint scale-$(k-|a|-1)$ boxes contained in $B_k(a)$.
A hierarchy is said to be \textit{achievable} if for all boxes $B_k(a)$ its children are a pair of boxes $(B,B')$ from the choice of pairs given in
Lemma~\ref{lema:cascading_half_cross}.
Finally, let us define
\begin{equation*}
    \cX_k := \{\cH_k;\; \text{$\cH_k$ is an achievable hierarchy}\}.
\end{equation*}
It is clear from Lemma~\ref{lema:cascading_half_cross} that
$\# \cX_k \le C \cdot C^2 \cdot \ldots \cdot C^{2^{k-1}} \le C^{2^{k}}$ for
$C=C(d,\alpha,\beta)$ given by \eqref{e:entropy}.

Define the set of \textit{leaves} of $\cH_k$ as
\begin{equation*}
    L(\cH_k) := \{B \in \cH_k;\; B = B_k(a), \ a \in \{0,1\}^{k}\},
\end{equation*}
that is, the set of scale-$0$ boxes of $\cH_k$.
Then, for any probability measure given by a GCP we have
\begin{equation}
\label{eq:hierarchy_estimate}
\PP(H(B_k))
    \le \sum_{\cH_k \in \cX_k}
        \PP\Bigl(
            \bigcap_{B \in L(\cH_k)} \hspace{-.3cm} H(B)
            \Bigr)
    \le C^{2^{k}} \max_{\cH_k}
        \PP\Bigl(
            \bigcap_{B \in L(\cH_k)} \hspace{-.3cm} H(B)
            \Bigr).
\end{equation}

Recall that we want to prove that the infection dies out almost surely.
Since $\PP(\tau^{\{0\}} = \infty) \le \PP(H(B_k))$, the estimate
in~\eqref{eq:hierarchy_estimate} shows it is sufficient to prove
that $\PP\bigl(\cap_{B \in L(\cH_k)} H(B)\bigr) \le \varepsilon^{2^{k}}$
for $\varepsilon$ sufficiently small, uniformly over $\cH_k \in \cX_k$.

Fix any achievable hierarchy $\cH_k \in \cX_k$. By construction, it contains
$2^{k}$ scale-$0$ boxes (its leaves) and any two of them are either
separated by a spatial distance of at least $(\alpha/2-2)2l_0$ or by a
temporal distance of at least $(\beta/2-2)h_0$.
Indeed, we have for $a \neq b$ with $|a| = |b| = k$ that $B_k(a)$ and $B_k(b)$ are both contained in $B_k(a
\wedge b)$, but are in different children of $B_k(a \wedge b)$. Hence, by
definition of achievable hierarchy boxes $B_k((a \wedge b)0)$ and
$B_k((a \wedge b)1)$ are well-separated, implying that $B_k(a)$ and $B_k(b)$
enjoy the same property.

Fix some ordering $\{L_1, L_2, \ldots, L_{2^{k}}\}$ for $L(\cH_k)$ such
that $L_j$ is always either above or at the same height of every previous leaf.
Since
\begin{equation}
\label{eq:prod_conditionals}
\PP \Big( \medmath{\bigcap\limits_{j=1}^{2^k}} H(L_j)\Big)
    = \prod_{j=1}^{2^{k}} \PP\big(H(L_j) \mid H(L_1) \cap \ldots \cap H(L_{j-1})\big),
\end{equation}
we focus on estimating the conditional probabilities above. This task depends
on the point processes chosen for the GCP.

\section{Contact Process on Dynamic Edges}
\label{sec:cpde}

The authors in~\cite{LR} define the CPDE on $\mathbb{Z}$, but also remark that
the definition extends easily to any connected graph with bounded degree. Here
we will focus on the case of $\ZZ^{d}$. The idea is to start with an
underlying dynamic \textit{environment} process $\zeta_t \in
\{0,1\}^{E(\ZZ^{d})}$ and, conditional on the realization of this environment
to define an \textit{infection} process $\eta_t \in \{0,1\}^{\ZZ^{d}}$ similar
to the classical contact process, with the main difference that its evolution
depends on the changing environment. More precisely, let us fix two parameters
$v>0$ and $p \in (0,1)$. Independently of everything else, each edge $e$ in
the environment is assigned an initial state $\zeta_0(e)$ with the Bernoulli
distribution with parameter $p$, $\Ber(p)$, and independently updates its
state as follows:
\begin{alignat*}{3}
    & 0 \longrightarrow 1 \quad && \text{at rate} \quad && vp, \\
    & 1 \longrightarrow 0 \quad && \text{at rate} \quad && v(1-p).
\end{alignat*}
We say that the edge $e$ is open at time $t$ if $\zeta_t(e) =1$ and that it is
closed at time $t$ otherwise.  We may think of the state of each edge $e$
independently alternating between open and closed at the given rates.  This
defines the Markov process $\{\zeta_t\}_{t\geq 0}$ taking values in
$\{0,1\}^{E(\mathbb{Z}^d)}$ usually called dynamic bond percolation on
$\mathbb{Z}^d$ with density parameter $p$ and rate $v$.  In fact, the choice for the
initial distribution $\zeta_0$ as being the product of $\Ber(p)$ implies that
$\zeta_t$ is stationary. Hence, for each fixed time $t$, the configuration
$\{\zeta_t(e)\}_{e\in E(\mathbb{Z}^d)}$ is distributed as an independent bond
percolation process on $\mathbb{Z}^d$.

We will now define the contact process $\eta_t$ whose evolution will depend on
a parameter $\lambda>0$ and on the underlying environment $\zeta_t$.  At each
site $x$, the state $\eta_t(x)$ evolves as follows:
\begin{alignat}{3}
    & 1 \longrightarrow 0 \quad && \text{at rate} \quad && 1, \nonumber \\
    & 0 \longrightarrow 1 \quad && \text{at rate} \quad &&
    \lambda \sum_{y \sim x} \zeta_t(xy) \eta_t(y).
  \label{e:def_eta}
\end{alignat}
In words, each site $y$ attempts to infect a neighboring site $y$ at rate
$\lambda$ through the edge $e=xy$ as in the usual contact process on
$\mathbb{Z}^d$. However, it will only succeed in case that edge is found open
at the time of the attempt.

In \cite[Section 3.1]{LR} the authors define this process via a standard
graphical construction, employing a collection of independent Poisson point
processes on $(0,\infty)$:
\begin{itemize}
\item $\{\cO_e\}$ of rate $vp$, whose marks provide the opening times of
    edge $e$.
\item $\{\cC_e\}$ of rate $v(1-p)$, whose marks provide the closing times of
    edge $e$.
\item $\{\cI_e\}$ of rate $\lambda$, whose marks provide the times of potential
    transmissions along edge $e$.
\item $\{\cR_x\}$ of rate $1$, whose marks provide the cure (recovery) times of
    site $x$.
\end{itemize}
This graphical construction serves as a tool in our methods. The reader
may consult \cite[Section 3.1]{LR} for more details.

\begin{remark} 
It is worth noticing that actually the CPDE can be seen as a Renewal Contact
Process with renewals on the edges and a delay.
However, it is not straightforward how to exploit this fact since the resulting
interarrival distribution $\mu$ depends on parameters $\lambda,v,p$, and is
therefore affected by changes on any of these parameters.
\end{remark}

Using the above notation, the critical parameter is given by
\begin{equation*}
\lambda_0(v,p)
    = \inf\{\lambda > 0;\;
    \PP_{v,p,\lambda} (\smash{\eta^{\{0\}}_t} \not\equiv 0,\ \forall t > 0) > 0\}.
\end{equation*}

\subsection{Proof of Theorem~\ref{teo:extinction_cpde}(i)}
\label{sub:teo_cpde_i}

In this section we apply the hierarchical approach to renormalization. On the
course of the proof we will need to use a straightforward estimate on how long
the usual contact process restricted to a finite (static) cluster can survive.
We state it as a lemma and include its proof for the reader's convenience.
\begin{lema}
\label{lema:extinction_cp}
Let $G$ be any connected subgraph of $\ZZ^{d}$ with at most $n$
vertices and $\tau = \tau(G)$ the extinction time of the contact process
started from the configuration $\zeta^G_0$, where only the vertices in $G$ are
infected. There exists $\kappa = \kappa(\lambda, d)>0$ such that for every $\nu > 1/n$,
\begin{equation}
\label{eq:extinction_cp}
    \PP(\tau(G) \ge e^{\nu n}) \le \exp[- e^{(\nu - \kappa) n}/2].
\end{equation}
\end{lema}

\begin{proof}
We fix $G$ throughout the proof and provide bounds that are uniform in $G$.
Let $T_j$ be the event in which every vertex in $G$ heals before the first
transmission in $[j,j+1]$. Then, it is clear that
$\{\tau \ge k\} \subset \cap_{j=0}^{k-1} T_j^{\comp}$, which implies for
$k = \lfloor e^{\nu n} \rfloor$,
\begin{equation*}
\PP(\tau \ge e^{\nu n})
    \le \prod_{j=0}^{k-1} \PP(T_j^{\comp})
    = (1 - \PP(T_0))^{k}
    \le \exp[ - k \cdot \PP(T_0) ]
      \leq \exp[- \PP(T_0) \,e^{\nu n}/2],
\end{equation*}
where we have used the fact that the events $T_j$ are independent and have
the same probability together with the fact that
$\lfloor e^{\nu n} \rfloor > e^{\nu n}/2$ if $\nu > 1/n$.
Since $G$ has at most $2dn$ edges and $n$ vertices, it is clear that
\begin{align*}
\PP(T_0)
    &\ge \PP\Bigl(
    \bigcap_{e \in E(G)} \{\cI_{e} \cap [0,1] = \emptyset\} \cap
    \bigcap_{v \in V(G)} \{\cR_{v} \cap [0,1] \neq \emptyset\}
    \Bigr)
    \ge e^{-2d \lambda n} \cdot (1 - e^{-1})^{n} \\
    &\ge \exp\bigl[- \bigl(2d \lambda + 1\bigr) n\bigr],
\end{align*}
using that $1-e^{-1}>e^{-1}$. Taking $\kappa(\lambda, d) :=  2d \lambda + 1$,
the inequality~\eqref{eq:extinction_cp} follows.
\end{proof}

\begin{proof}[Proof of Theorem \ref{teo:extinction_cpde}(i)]
Let $p < p_c(d)$ and $\lambda > 0$ be fixed.  Our goal is to show that for $v$
small, depending on $d$, $p$ and $\lambda$, the CPDE dies out. In order to
apply our renormalization approach, we need to define the sequence of
scales $l_k$ and $h_k$ as in \eqref{e:scales}. Recall that they become
fully determined once we choose the values for $\alpha$, $\beta$, $l_0$ and
$h_0$.  We start by fixing $\alpha=4$. The other values will be determined
next depending on $p$, $\lambda$ and $d$.

Let us write $\delta = \delta(p) := \frac{1}{4}(p_c(d) - p)$ and fix
$\beta = \beta(p, d) \geq 6$ sufficiently large so that
\begin{equation}
\label{eq:beta_delta_choice}
e^{- (\beta/2 - 2) \delta} < \delta.
\end{equation}

Having fixed $\alpha$ and $\beta$ it only remains to chose $l_0$ and $h_0$ suitably.
In the following, we take
\begin{equation}
v := \delta/h_0
\label{e:choice_v}
\end{equation}
so that $v$ will be determined once $h_0$ has been chosen.

For a scale-$k$ box $B_k = [-l_k,l_k]^{d} \times [0,h_k]$ and a hierarchy
$\cH_k \in \cX_k$, label the leaves in $\cH_k$ as $L_1, \ldots, L_{2^{k}}$
in such a way that leaves located higher in time are assigned greater
indices.  Fix some leaf $L_j$ of the form $\pi(L_j) \times [s_j, s_j + h_0]$
(recall the notation $\pi(L_j)$ for its space projection). Also
recall that $\cO_e$ and $\cC_e$ are PPPs whose arrivals represent the times
at which the edge $e$ opens and closes, respectively.  We say that an edge
$e$ with both endvertices in $\pi(L_j)$ is $L_j$-\textit{available} if at
least one of the following conditions is satisfied:
\begin{enumerate}[(i)]
\item $e$ opens during the time interval associated to $L_j$:
    $\cO_e \cap [s_j, s_j+h_0] \neq \emptyset$;
\item $e$ does not update in the time interval of length
    $(\beta/2-2) h_0$ prior to the time interval of $L_j$: \\
    $(\cO_e \cup \cC_e) \cap [s_j- (\beta/2-2) h_0, s_j] = \emptyset$;
\item $e$ updates during $[s_j- (\beta/2-2) h_0, s_j]$,
    and $e$ is open at time $s_j$.
\end{enumerate}
Hence we have
\begin{align}
\PP(\text{$e$ is $L_j$-available})
    &\le (1 - e^{- pv h_0}) + (e^{- v (\beta/2-2) h_0}) +
        \PP(\zeta_{s_j}(e) = 1) \nonumber\\
    &= (1 - e^{- p\delta}) + (e^{- (\beta/2-2)\delta}) + p \nonumber\\
    &\le p\delta + e^{- (\beta/2 - 2) \delta} + p \nonumber\\
\label{eq:prob_available_edge}
    &< \frac{1}{2} (p+p_c(d)),
\end{align}
where we used \eqref{e:choice_v} in the second line and the last inequality is
due to our choices of $\beta$ and $\delta$ in~\eqref{eq:beta_delta_choice}.

Consider the graph whose vertices are sites in $\pi(L_j)$ and whose edges are
those that are $L_j$ available.  Let us call $C_j(x)$ the cluster
containing the vertex $x \in\pi(L_j)$ in this graph. Notice that $C_j(x)$
is either equal to $\{x\}$ or is an open cluster of a Bernoulli bond
percolation process in $\pi(L_j)$  with parameter at most $(p+p_c(d))/2$,
hence subcritical. By the exponential decay of the cluster size
distribution (cf.~\cite[Theorem~(6.75)]{Gri})
\begin{equation}
\label{e:def_psi}
\text{$\exists\, \psi = \psi(p, d)>0$ \quad s.t.\ \quad $\PP(|C_j(x)| \geq m)
    \le e^{- \psi m}\quad \forall m \in \NN$}.
\end{equation}
For each $e=xy$ with both endvertices $x,y \in \pi(L_j)$ and $t \in [s_j, s_j+h_0]$
let us define
\begin{equation}
\label{e:def_zeta_hat}
\hat{\zeta}_{j,t}(e) := \I_{\{\text{$e$ is $L_j$-available}\}}.
\end{equation}
Using the graphical construction in terms of the point processes $\mathcal{O}$,
$\mathcal{C}$, $\mathcal{I}$ and $\mathcal{R}$ we can define inside $L_j$ the
process $\hat{\eta}_{j,t}$ where, the initial configuration is given by
$\hat{\eta}_{j,s_j} (x)= 1$ for every $x \in \pi(L_j)$ and instead of
$\zeta_t(xy)$ one uses $\hat{\zeta}_{j,t}(xy)$ in \eqref{e:def_eta}. Roughly
speaking, replacing $\zeta$ by $\hat{\zeta_j}$  amounts to enlarging the open
clusters at the basis of $L_j$ and then to keep then frozen for time $h_0$.

We say that the leaf $L_j$ is good if the half-crossing event occurs inside
$L_j$ for the process $\hat{\eta}_j$. Let us denote this event by
$\hH(L_j)$. Notice that for different leaves, these events depend on
disjoint regions of the Poisson point processes $\mathcal{O}$, $\mathcal{C}$,
$\mathcal{I}$ and $\mathcal{R}$ so they are independent.

The point in considering these events is that $H(L_j) \subset \hH(L_j)$.
Therefore,
\begin{equation*}
\PP \Big( \medmath{\bigcap\limits_{j=1}^{\smash{2^k}}} H(L_j)\Big)
    \leq \PP \Big( \medmath{\bigcap\limits_{j=1}^{\smash{2^k}}} \hH(L_j)\Big)
    = \PP \big( \hH(L_1)\big)^{2^k},
\end{equation*}
where we have used independence and translation invariance. In view of
\eqref{eq:hierarchy_estimate}, it suffices to prove that $\PP ( \hH(L_1)) <
1/(2 C(d, \alpha, \beta))$ where $C(d, \alpha, \beta)$ has been fixed in Lemma
\ref{lema:cascading_half_cross}. This will be done by suitably choosing $l_0$
and $h_0$.

By \eqref{e:def_psi} the probability of the event
\begin{equation*}
\label{eq:defi_large_clusters}
U_j := \big\{\exists x \in \pi(L_j);\; |C_j(x)| \ge (2d/\psi) \ln l_0\big\}
\end{equation*}
is bounded by
\begin{equation}
\label{eq:bound_large_clusters}
\PP(U_j)
    \le c(d) l_0^{d} \cdot e^{- \psi (2d/\psi) \ln l_0}
    = c(d) l_0^{-d}.
\end{equation}
Since every infection path inside $L_j$ must only jump through $L_j$-available
edges, each of these paths is contained in a cluster $C_j(x)$, that is
typically much smaller than $\pi(L_j)$.

Let us now fix $l_0$ sufficiently large so that
\begin{equation}
\label{e:assumption_l_0}
\frac{2d}{\psi}\ln l_0 < l_0
\quad \text{ and } \quad
l_0 \geq {[4 c(d) C(d, \alpha, \beta)]}^{1/d}.
\end{equation}
Then, estimate~\eqref{eq:bound_large_clusters} implies
\begin{equation}
\label{eq:bound_large_clusters_2}
\mathbb{P}(U_j) \leq \frac{1}{4 C(d, \alpha, \beta)}.
\end{equation}
Moreover, any spatial half-crossing for $\hat{\eta}_j$ inside $L_j$ has to
traverse at least $l_0$ edges. Therefore, the occurrence of such
half-crossings implies the occurrence of $U_j$.

On $U_j^{\comp}$ we know that all of the available clusters in $\pi(L_j)$ are
small, that is, each $C_j(x)$ contains at most $(2d/\psi) \ln l_0$ sites. In
order for a temporal half-crossing to occur the process $\hat{\eta}_j$ must
survive for time at least $h_0/2$ in one of these small clusters.

Let $\nu = \nu(p, \lambda, d, l_0) =\max\{ \psi +\kappa, (2d/\psi) \ln l_0\}$
where $\kappa=\kappa(\lambda, d)$ is given in Lemma \ref{lema:extinction_cp}
and $\psi(p,d)$ is given in \eqref{e:def_psi}. Define
\begin{equation*}
V_j:= \big\{\exists x \in \pi(L_j);\;
    \text{$\hat{\eta}_j$ survives longer than $e^{\nu (2d/\psi) \ln l_0}$
    inside $C_j(x)$}\big\}.
\end{equation*}
By Lemma~\ref{lema:extinction_cp},
\begin{equation}
\label{eq:temp_cross_cpde}
\PP(U^{\comp} \cap V_j)
    \le c(d)l_0^{d} \cdot e^{- (\nu - \kappa) (2d/\psi) \ln l_0}
    = c(d)l_0^{d- (\nu - \kappa) (2d/\psi)}
    \le c(d)l_0^{-d}
    \le \smash{\frac{1}{4 C(d, \alpha, \beta)}}.
\end{equation}

Therefore, uniformly over
\begin{equation}
\label{e:def_h0}
h_0
    \ge 2 \cdot e^{\nu (2d/\psi) \ln l_0}
    = 2 \cdot \smash{l_0^{\nu (2d/\psi)}}
\end{equation}
the following bound holds
\begin{equation}
\label{eq:bound_half_cross_cpde}
\PP(\hH(L_j))
    \le \PP(U) + \PP(U^{\comp} \cap V)
    \le \frac{1}{2 C(d, \alpha, \beta)},
\end{equation}
as it can be seen by just plugging~\eqref{eq:bound_large_clusters}
and~\eqref{eq:temp_cross_cpde}. Thanks to \eqref{e:choice_v} this implies that
any choice of ${v \in (0, \delta \smash{l_0^{-\nu (2d/\psi)}}/2)}$ is sufficient to
establish~\eqref{eq:bound_half_cross_cpde}. This finishes the proof with
$v_0 = l_0^{-\nu (2d/\psi)}/2$.
\end{proof}

\subsection{Proof of Theorem~\ref{teo:extinction_cpde}(ii)}
\label{sub:teo_cpde_ii}

For a bond percolation configuration~$\zeta$ in~$\ZZ^d$ and any connected
subgraph~$B$ of~$\ZZ^d$, we denote by~$G_\zeta(B)$ the random subgraph of~$B$
induced by the open bonds in~$\zeta$. We denote by~$G^*_\zeta(B)$ the connected
component with largest cardinality of~$G_\zeta(B)$ (we can adopt some arbitrary
procedure to decide between components in the case of a tie). The following
result follows from Proposition~3.2 in~\cite{SV} (which in turn is proved using
results from~\cite{CM} and~\cite{P}).

\begin{prop}
\label{prop:schapira}
Assume that~$d \ge 2$,~$p > p_c(\ZZ^d)$ and~$\zeta$ is sampled from the product
Bernoulli($p$) distribution. Then, there exists~$\delta > 0$ such that the
following holds for~$n$ sufficiently large. With probability higher
than~$1-\exp\{-(\log n)^{1+\delta}\}$, the component
$G^*_\zeta(\{0,\ldots,n-1\}^d)$ has  cardinality larger
than~$n^{d-\frac14}$, and all other components of~$G_\zeta(\{0,\ldots,
n-1\}^d)$ have cardinality smaller than~$n^{d-\frac12}$.
\end{prop}

For the rest of this section we fix~$d \ge 2$ and~$\bar{p} > p_c(\ZZ^d)$.
Since~$p \mapsto \lambda_0(v,p)$ is non-increasing, we will
establish~\eqref{eq_sup_vp} once we prove that
\begin{equation}
\label{eq_sup_vp2}
    \sup\left\{\lambda_0(v,\bar{p}): v \ge 0\right\} < \infty.
\end{equation}
Moreover, by Theorem~2.3 in~\cite{LR} we have that~$\lambda_0(v,\bar{p})$
converges to a finite limit as~$v \to \infty$. Hence,~\eqref{eq_sup_vp2} will
follow from showing that
\begin{equation}
\label{eq_sup_vp3}
	\sup\left\{\lambda_0(v,\bar{p}): 0 \le v \le \bar{v}\right\} < \infty
    \qquad \text{for all } \bar{v} > 0.
\end{equation}
Hence, for the rest of this section we fix~$\bar{v} > 0$ and we will prove
that~\eqref{eq_sup_vp3} holds. The dynamic environment~$\{\zeta_t\}_{t\ge 0}$
will have edge density parameter equal to~$\bar{p}$ and edge update speed~$v
\in [0,\bar{v}]$ which will be clear from the context or irrelevant. 

For each~$n \in \NN$ and~$z \in \ZZ$, define
\begin{align*}
    &B_n(z) := zn\vec{e}_1 + \{0,\ldots, n-1\}^d,\\
    &B_n'(z) := B_n(z-1)\cup B_n(z)\cup B_n(z+1).
\end{align*}
Define the event
\begin{equation*}
{E}_n(z,t) := \left\{
    \begin{array}{l}G^*_{\zeta_t}(B_n'(z))
    \text{ is the unique component in~$G_{\zeta_t}(B_n'(z))$ that}\\
    \text{  intersects the three boxes~$B_n(z-1)$,~$B_n(z)$,~$B_n(z+1)$}
    \end{array}\right\}.
\end{equation*}
We then have

\begin{lema}
\label{lem:first_exp_log}
For~$n$ sufficiently large we have, for any~$z \in \ZZ$ and~$t \ge 0$,
\begin{equation*}
\PP({E}_n(z,t)) > 1-(2n+1)\cdot \exp\{-(\log n)^{1+\delta}\},
\end{equation*}
where~$\delta$ is given in Proposition~\ref{prop:schapira}.
\end{lema}

\begin{proof}
By translation invariance and stationarity, it suffices to prove the statement
with~$z = 0$ and~$t = 0$. For each box~$B \subset B_n'(0)$ of the form
\begin{equation*}
B = \{u,\ldots, u+n-1\}\times \{0,\ldots, n-1\}^d,\quad u \in \{-n,\ldots,n\},
\end{equation*}
let~$A(B)$ denote the event that~$G^*_{\zeta_0}(B)$ has
cardinality larger than~$n^{d-\frac{1}{4}}$, and all other components
of~$G_{\zeta_0}(B)$ have cardinality smaller than~$n^{d-\frac12}$. It is
easy to see that~$\cap_B A(B) \subset {E}_n(0,0)$. Since there are~$2n +1$
such boxes, Proposition~\ref{prop:schapira} and a union bound guarantee
that~$\PP(\cap_B A(B)) > 1-(2n+1)e^{-(\log n)^{1+\delta}}$.
\end{proof}

We now define the event
\begin{equation*}
E_n'(z,t)
    = \bigcap_{s \in [t,t+1]} E_n(z,s),
    \qquad n \in \NN,\;z \in \ZZ,\;t \ge 0.
\end{equation*}
We then have
\begin{lema}
\label{lem:integral}
For any~$\varepsilon > 0$ there exists~$n_0$ such that if~$n \ge n_0$ we have
$\PP(E_n'(z,t))> 1-\varepsilon$ for all~$v \in [0,\bar{v}]$,~$z\in
\ZZ$ and~$t\ge 0$.
\end{lema}

\begin{proof}
Fix~$v,z,t$ as in the statement. Let~$\cT \subset [0,\infty)$ denote the set of
update times of the edges of~$B_n'(z)$. Then,~$\cT$ is a Poisson point
process whose intensity is smaller than to~$C_dvn^d$, for some constant~$C_d > 0$.
In particular, for any~$s \ge 0$ we have
\begin{equation*}
\PP(\cT \cap [s,s+(C_dvn^d)^{-1}] = \varnothing)
    \ge \PP(\text{Poisson}(1) = 0)
    = e^{-1}.
\end{equation*}
On the event~$(E_n'(z,t))^c$, let~$\tau$ denote the smallest~$s \in [t,t+1]$
such that $(E_n(z,s))^c$ occurs. Letting~$A = (E_n'(z,t))^c \cap \{\cT\cap
(\tau,\tau+(C_dvn^d)^{-1}] = \varnothing\}$, we have, by the strong Markov
property,
\begin{equation*}
\PP(A) \ge \PP((E_n'(z,t))^c) \cdot e^{-1}.
\end{equation*}
Now, noting that
\begin{align*}
(C_dvn^d)^{-1}\cdot \I_A \le \int_t^{t+2}\I_{(E_n(z,s))^c}\;\mathrm{d}s
\end{align*}
and taking expectations, we obtain
\begin{align*}
(C_dvn^{d})^{-1}\cdot e^{-1}\cdot \PP((E_n'(z,t))^c)
    &\le \EE\left[\int_t^{t+2} \I_{(E_n(z,s))^c}\;\mathrm{d}s \right] \\
    &\le 2\cdot (2n+1)\cdot \exp\{-(\log n)^{1+\delta}\},
\end{align*}
where the last inequality follows from Lemma~\ref{lem:first_exp_log} and
Fubini's theorem. We thus have
\begin{equation*}
\PP((E_n'(z,t))^c)
    \le 2eC_dvn^d(2n+1)\cdot \exp\{-(\log n )^{1+\delta}\},
\end{equation*}
and the right-hand side can be made as small as desired by taking~$n$ large,
uniformly in~$v \in [0,\bar{v}]$.
\end{proof}

We now give some further definitions. A finite sequence~$\gamma=(x_0,\ldots,
x_m)$ of vertices of~$\ZZ^d$ is called a self-avoiding path
if~$x_0,\ldots,x_m$ are all distinct and~$x_i \sim x_{i+1}$ for each~$i$. For
such a sequence~$\gamma$ and~$t>s\ge 0$, we let~$\Phi(\gamma,s,t)$ denote the
indicator function of the event that, in the graphical construction of the
process, there exist times~$s < t_1<\ldots < t_m < t$ such that, for each~$i$,
there is a transmission mark in the edge~$\{x_{i-1},x_{i}\}$ at time~$t_i$. We
emphasize that this definition does not involve the recovery marks or the edge
percolation environment, but only the transmission marks.

For a finite connected subgraph~$B$ of~$\ZZ^d$, let~$\Gamma_B$ denote the
(finite) set of all self-avoiding paths contained in~$B$.
\begin{lema}
\label{lem:fast_infection}
Let~$t> 0$ and~$B$ be a finite connected subgraph of~$\ZZ^d$. Then, for
any~$\varepsilon > 0$, we can take~$\lambda$ large enough so that
\begin{equation*}
\PP(\Phi(\gamma,s,s+t)
    = 1\text{ for all }\gamma \in \Gamma_B) > 1-\varepsilon
    \quad \text{for all } s \ge 0.
\end{equation*}
\end{lema}

\begin{proof}
By translation invariance, it suffices to treat~$s = 0$. We divide~$[0,t]$
into~$|\Gamma_B|$ sub-intervals of equal lengths and disjoint interiors.
The event in question is achieved if each edge of~$B$ has a transmission
mark in the interior of each of these sub-intervals. This has probability
as high as desired when~$\lambda \to \infty$.
\end{proof}

We now define a further event~$F_n(z,t)$ for~$n \in \NN$,~$z \in \ZZ$ and~$t
\ge 0$, as follows. Let~$t_1 < t_2 < \cdots < t_N$ denote, in increasing order,
the (random) times within the time interval~$[t,t+1]$ at which there is either
an edge  update or a recovery mark inside~$B_n'(z)$. Also let~$t_0 = t$
and~$t_{N+1} = t+1$. Then,~$F_n(z,t)$ is defined as the event that
\begin{equation*}
\Phi(\gamma,t_i,t_i+1) = 1
    \text{ for all }\gamma \in \Gamma_{B'_n(z)}
    \text{ and all }i \in \{0,\ldots, N\}.
\end{equation*}
In words, this is the event that, between two successive times~$t_i,t_{i+1}$
(each of which can correspond to a recovery mark or an edge update
inside~$B_n'(z)$), every self-avoiding path inside~$B_n'(z)$ can be traversed
by following transmissions. This guarantees that, if at time~$t_i$ one of the
vertices of~$G^*_{\zeta_{t_i}}(B_n'(z))$ is infected, then immediately before
time~$t_{i+1}$ \textit{all} vertices of this cluster will be infected.

\begin{lema}
\label{lem:in_between}
For any~$n \in \NN$ and~$\varepsilon > 0$, there exists~$\lambda' > 0$ such that
\begin{equation*}
\PP(F_n(z,t)) > 1-\varepsilon
    \qquad \text{ for any }\lambda > \lambda',
    \;v \in [0,\bar{v}],\; z \in \ZZ^d,\text{ and }t \ge 0.
\end{equation*} 
\end{lema}

\begin{proof}
Fix~$n \in \mathbb{N}$ and~$\varepsilon > 0$. By translation invariance, it is
sufficient to prove that there exists~$\lambda' > 0$ such
that~$\PP(F_n(0,0)) > 1-\varepsilon$ for any~$\lambda > \lambda'$
and~$v \in [0,\bar{v}]$. Let~$t_1<\cdots < t_N$ denote the times in~$[0,1]$
at which there is either an edge update or recovery mark inside~$B_n'(0)$,
and let~$t_0 = 0$ and~$t_{N+1}= 1$. Let
\begin{equation*}
X:=\inf\{|t_{i+1}-t_i|:\;i\in\{0,\ldots, N\}\}.
\end{equation*}

It is easy to see that there exists~$\delta > 0$ such that
\begin{equation*}
\PP(X > \delta,\; N < 1/\delta) > 1-\frac{\varepsilon}{2}
    \quad \text{for any } v \in [0,\bar{v}].
\end{equation*}
Now, by Lemma~\ref{lem:fast_infection} and a union bound, we can
obtain~$\lambda'>0$ such that, for any~$\lambda > \lambda'$ and
any~$v \in [0,\bar{v}]$,
\begin{equation*}
\PP(F_n(0,0)\mid X > \delta,\; N < 1/\delta) > 1-\frac{\varepsilon}{2},
\end{equation*}
completing the proof.
\end{proof}

Finally, define
\begin{equation*}
\eta_k(z)
    := \I[{E_n'(z,k)\cap F_n(z,k)}],\qquad z \in \ZZ,\;k\in\NN_0.
\end{equation*}
It will be useful to note the following:
\begin{claim}
\label{cl:infinite_path}
If~$\eta_k(z) = 1$ and at least one site of~$G^*_{\zeta_k}(B_n'(z))$ is
infected at time~$k$, then  all sites  of~$G^*_{\zeta_{k+1}}(B_n'(z))$ are
infected at time~$k+1$.
\end{claim}

\begin{proof}
Let~$t_1 < t_2 < \cdots < t_N$ denote  the times within~$[k,k+1]$ at which
there is either an edge  update or a recovery mark inside~$B_n'(z)$. The
definition of~$F_n(z,k)$ guarantees that, in~$[k,t_1)$, the
component~$G^*_{\zeta_k}(B_n'(z))$ becomes fully infected. The definition
of~$E'_n(z,k)$ guarantees that the components~$G^*_{\zeta_{t_1-}}(B_n'(z))$
and $G^*_{\zeta_{t_1}}(B_n'(z))$ have in common a component that intersects
the three boxes~$B_{n}(z-1)$,~$B_n(z)$ and~$B_n(z+1)$. In particular, at
least one infection remains in~$G^*_{\zeta_{t_1}}(B_n'(z))$. Proceeding
recursively, we obtain the result.
\end{proof}

\begin{proof}[Proof of Theorem~\ref{teo:extinction_cpde}(ii)]
Fix~$\varepsilon > 0$. Assume that~$n$ is large enough, as required by
Lemma~\ref{lem:integral}, and then assume that~$\lambda$ is large enough,
as required by Lemma~\ref{lem:in_between}. These choices guarantee
that~$\PP(\eta_k(z) =1) > 1-2\varepsilon$ for all~$\lambda > \lambda'$,~$v
\in [0,\bar{v}]$,~$k\in \mathbb{N}$ and~$z\in\ZZ$. We also have that
if~$(k,z)$ and~$(k',z')$ have either~$k\neq k'$ or~$|z-z'| > 2$,
then~$\eta_k(z)$ and~$\eta_{k'}(z')$ are independent. Hence,~$(\eta_k(z))$
dominates a one-dependent Bernoulli  field with density
above~$1-2\varepsilon$. If~$\varepsilon$ is sufficiently small, then with
positive probability there is an infinite sequence~$0=z_0,z_1,\ldots \in
\ZZ$ such that~$|z_k - z_{k+1}|\le 1$ and~$\eta_k(z_k) = 1$ for every~$k$.
By Claim~\ref{cl:infinite_path}, we obtain that there is an infinite
infection path contained in the space-time set~$\cup_k (B_n'(\eta(k))
\times [k,k+1])$. This proves that~$\PP(B_n'(0) \rightsquigarrow \infty) >
0$; since
\begin{equation*}
\PP(B_n'(0) \rightsquigarrow \infty)
    \le |B_n'(0)|\cdot \PP((0,0)\rightsquigarrow \infty),
\end{equation*}
it follows that~$\PP((0,0)\rightsquigarrow \infty) > 0$.
\end{proof}

\section{Edge Renewal Contact Process}
\label{sec:ercp}




\subsection{Uniform control for renewals}
\label{sub:uniform_control_for_renewals}

Our study of ERCP is based on a uniform control for the probability of having
renewal marks in an interval of fixed length. The next lemma summarizes
inequalities that achieve this goal. These estimates are in the core of all
subsequent computations and justify our hypotheses on $\mu$ and $\nu$.

\begin{lema}[Uniform estimates]
\label{lema:unif_estimates}
Let $\mu$ be any probability distribution on $\RR_+$ and let $\cR$ be a renewal
process with interarrival $\mu$ started from some $\tau \le 0$.
\begin{enumerate}[(i)]
\item If $f: [0, \infty) \to [0,\infty)$ is non-decreasing,
    $\lim_{x \to \infty}f(x)= \infty$, and
    ${\int x f(x) \,\mu(\mathrm{d}x) < \infty}$, then uniformly on $\tau$ we have
    \begin{equation}
    \label{eq:moment_condition}
    \sup_{t \geq 0} \PP( \cR \cap [t, t + h] = \emptyset)
        \le \frac{C}{f(h)},
    \end{equation}
    for some positive constant $C = C(\mu, f)$ whenever $f(h) > 0$. Moreover,
    if $\int x \,\mu(\mathrm{d}x) < \infty$ then given any $\varepsilon > 0$
    there is $h_0=h_0(\varepsilon) > 0$ such that uniformly on $\tau$, we have
    \begin{equation}
    \label{eq:finite_1st_moment}
    \sup_{t \geq 0} \PP( \cR \cap [t, t + h_0] = \emptyset)
        \le \varepsilon.
    \end{equation}

\item If $\mu$ is continuous, then given $\varepsilon>0$ there is
    $w_0=w_0(\varepsilon) > 0$ such that uniformly on $\tau $  we have
    \begin{equation}
    \label{eq:uniform_cont}
    \sup_{t \geq 0} \PP( \cR \cap [t, t + w_0]\neq \emptyset)
        \le  \varepsilon.
    \end{equation}
\end{enumerate}
\end{lema}

\begin{proof}
We can assume that $\tau = 0$, since taking $\tau<0$ is equivalent to taking
the supremum over $t \ge -\tau$.

The first statement in (i) is exactly Lemma~2.3 of~\cite{FMUV}.
We also notice that the inequality~\eqref{eq:finite_1st_moment} is a straightforward consequence
of~\eqref{eq:moment_condition}. Indeed, it suffices to show that when
$\int x \,\mu(\mathrm{d}x) < \infty$ one can find a function $f$ satisfying
the requirements for~\eqref{eq:moment_condition}.
Finding such function $f$ is a standard analysis exercise
and we omit the proof.

The proof of (ii) is based on the fact that when $\mu$ is continuous its
renewal function $U(t)$ is uniformly continuous on $\RR_+$. The continuity
of $U(\cdot)$ follows at once from that of $\mu$. To ensure uniform
continuity, we have to control the behavior of $U(t)$ as $t \to \infty$.
This follows from the classical renewal theorem, which implies $\lim_{t \to
\infty} (U(t+h) -U(t))=\frac{h}{\int x \; \mu(\mathrm{d}x)}$ (understood as
zero if the integral diverges).
Hence, given $\varepsilon >0$ there is $w_0=w_0(\varepsilon)$ such that
\begin{equation*}
\sup_{t \ge 0} \PP( \cR \cap (t, t+w_0] \neq \emptyset)
    \le \sup_{t \ge 0} \bigl(U(t+w_0) - U(t)\bigr)
    \le \varepsilon. \qedhere
\end{equation*}
\end{proof}

\subsection{Growth of ERCP}
\label{sub:growth_ercp}

We start this section by showing that if $\mu$ is continuous then a.s.\
the infection cannot reach infinitely many sites in finite time. When
considering First Passage Percolation, it is known that having a finite speed
is equivalent to $\mu(\{0\}) < p_c = p_c(\ZZ^{d})$, cf.~\cite{Kesten}. For
ERCP, this is not the case. For instance, if $\mu$ has an atom at $t\ge 0$ with
$\mu(\{t\}) >p_c$ then the cluster at time $t$ is a.s.\ infinite. Moreover, the
same phenomenon can be obtained by combining atoms at different times: e.g., if
$\mu(1)^2 + \mu(2) > p_c$ then we have the same problem, since it implies
$\PP(\cR \ni 2) > p_c$.

Our strategy to bound the speed of growth in an ERCP without cures is to make a
comparison with a toy model of iterated percolation. The idea is the following.
Fix $p < p_c(\ZZ^{d})$ and let $\cP_{i}$ be a family of independent Bernoulli
bond percolation models on $\ZZ^{d}$. Moreover, for $V \subset \ZZ^{d}$ let us
denote by $\cC_{i}(V)$ the connected component of $V$ in $\cP_{i}$ by open
edges. Given an initial finite non-empty set $C_0$, we define an increasing
sequence of sets by
\begin{equation*}
    C_n := \cC_n(C_{n-1}), \quad \text{for every $n \ge 1$}.
\end{equation*}

\medskip
\noindent
\textbf{Coupling.}
Fix $d \ge 1$ and a continuous distribution $\mu$ for the transmissions.
We compare iterated percolation with ERCP without cures. Assume that only the
origin is infected at time $0$. By Lemma~\ref{lema:unif_estimates}(ii) we
can choose an increasing sequence of times $(s_n)_{n\ge 0}$ with $s_0 := 0$
and $\lim s_n = \infty$ satisfying
\begin{equation}
\label{eq:times_sn}
\PP(\cR \cap [s_n, s_{n+1}] \neq \emptyset)
    < \frac{1}{2}p_c(\ZZ^{d}).
\end{equation}

Indeed, we can fix $\varepsilon = \frac{1}{4}p_c(\ZZ^{d})$ and define $s_n := n
w_0(\varepsilon)$. The sequence of times $(s_n)$ is important for the coupling
we describe next. Define $\cI_0$ as the set containing only the origin of
$\ZZ^{d}$. Notice that if a site $v$ is infected at time $s_1$ there must be a
sequence of sites $0 = x_0, x_1, \ldots, x_k = v$ such that $\cR_{x_{i-1} x_i}
\cap [s_0, s_1] \neq \emptyset$ for every $1 \le i \le k$.  Hence, we can find
all infected sites at time $s_1$ by exploring the connected component of the
origin in a canonical way: order the set of edges and always explore the
smallest edge that has not been explored yet but has some extremity in the
current infected cluster. This exploration produces a finite (random) set
$\cE_1$ of explored edges and finds all sites that have been infected till time
$s_1$.

Let us define $\cI_1$ as the set of all sites that are an extremity of some
edge in $\cE_1$. Notice that sites in $\cI_1$ may not be actually infected
(since we may have $e \in \cE_1$ with $\cR_e \cap [0,s_1] = \emptyset$), but we
consider them infected all the same.

We define sets $\cE_{n+1}$ and $\cI_{n+1}$ inductively. Given $\cI_n$, consider
an exploration process on edges of $\cE_n^{\comp}$ to find the infected cluster
at time $s_{n+1}$, starting from $\cI_n$ infected at time $s_n$. This consists
of checking the processes $\{\cR_e \cap [s_n, s_{n+1}];\; e \in
\cE_n^{\comp}\}$ until we determine the cluster. We define $\cE_{n+1}$ as the
union of all new explored edges with $\cE_n$ and define $\cI_{n+1}$ as the set
of all sites that are the extremity of some $e \in \cE_{n+1}$.

Since we only look at each edge at most once and different edges have
independent renewal processes, this construction is a minor modification of the
iterated percolation described above. Indeed, the set of explored edges in step $n$,
$\cE_{n} \setminus \cE_{n-1}$, is contained in the union of
$\cC_n(\cI_{n-1})$ with its external boundary of edges, a set we denote
$\bar{\cC}_{n}(\cI_{n-1})$. We conclude that it holds
\begin{equation*}
    \cI_{n} \subset \bar{\cC}_n(\cI_{n-1}), \quad \text{for every $n \ge 1$}
\end{equation*}
and, since in each step we have $\cI_{n} \setminus \cI_{n-1}$ is finite, the
infection cannot reach infinitely many sites in finite time.

\medskip
\noindent
\textbf{Iterated percolation growth.} We have just described a coupling in which
the growth of an iterated percolation model dominates the growth of ERCP. We
can actually use the coupling to estimate its rate of growth. We consider the
variation of iterated percolation that is relevant for us: given
$C_0 \subset \ZZ^{d}$ finite, define $C_n := \bar{\cC}_n(C_{n-1})$, for every
$n \ge 1$. The main quantity for us is
\begin{equation*}
R_n := \max\{\norm{x}_1;\; x \in C_n\}.
\end{equation*}
Having control on $R_n$, we are able to control $C_n$ since
$C_n \subset B(R_n)$. We are able to prove that the growth of
$R_n$ is very close to linear.
\begin{prop}
\label{prop:linear_Rn}
For any fixed $a > 1$ we have that almost surely, as $n \to \infty$
\begin{equation}
\label{eq:linear_Rn}
(p/2) \le \varliminf_n \frac{R_n}{n}
\quad \text{and} \quad
\varlimsup_n \frac{R_n}{n (\ln n)^{a}} = 0.
\end{equation}
\end{prop}

\begin{proof}
The first step of our proof is to show that $R_n$ must grow at least linearly.
This is quite straightforward, since in any step of the growth process we must
have some $x \in C_n$ that achieves $\norm{x} = R_n$ and an edge with extremity on
$x$ such that if it is open on $\cP_{n+1}$ then $R_{n+1} \ge R_n + 1$. This
shows $R_n$ dominates stochastically $R_0 + \Bin(n,p)$.
Hence, using Chernoff bounds we can write
\begin{equation*}
\PP\bigl(R_n \le (p/2)n\bigr)
    \le \PP\bigl(\Bin(n,p) \le (p/2)n\bigr)
    \le e^{- \frac{((p/2)n)^{2}}{2n}}
    = e^{- (p^2/8) n}.
\end{equation*}
Since $\sum_n \PP\bigl(R_n \le (p/2)n\bigr)$ converges, using the Borel-Cantelli
lemma we conclude that $R_n > (p/2)n$ eventually and the lower bound
in~\eqref{eq:linear_Rn} is proved. For the upper bound, we define events
\begin{equation*}
A_{n+1} := \{R_{n+1} \ge R_n + \eta \ln R_n\}.
\end{equation*}
for some constant $\eta(p,d) > 0$ that is chosen below.
Consider the filtration $\cF_n := \sigma(\cP_{i};\; i \le n)$ and notice that
$A_n \in \cF_n$. Given $\cF_{n}$ we have that on event $A_{n+1}$ there must
be some point $x \in \partial B(R_n)$ (notice that it does not need to
belong to $C_n$) that satisfies $x \leftrightarrow x + \partial B(\eta \ln R_n)$
in percolation $\cP_{n+1}$.  Hence, exponential decay of cluster size, see
e.g.~\cite[Theorem~(6.75)]{Gri}, gives the estimate
\begin{equation*}
\PP(A_{n+1} \mid \cF_{n})
    \le \sum_{x \in \partial B(R_n)}
        \PP(x \leftrightarrow x + \partial B(\eta \ln R_n) \mid \cF_{n})
    \le c R_{n}^{d-1} e^{- \psi(p,d) \eta \ln R_n}.
\end{equation*}
Choose $\eta(p,d) := \frac{d+1}{\psi(p,d)}$, which leads to
$\PP(A_{n+1} \mid \cF_{n}) \le c R_{n}^{-2}$.
The linear growth estimate says there is $n_{1}$ (random) such that
$R_n > (p/2)n$ for $n \ge n_1$, and we notice that
$x \mapsto c x^{-2}$ is decreasing for $n \ge n_1$.
This means that
\begin{equation*}
\sum_{n \ge n_1} \PP(A_{n+1} \mid \cF_{n})
    \le c(p,d) \sum_{n \ge n_1} n^{-2}
    < \infty.
\end{equation*}
Using a conditional Borel-Cantelli lemma,
see~\cite[Theorem~5.3.2]{Dur}, we have that $\PP(\varlimsup A_n) = 0$ implying
that there is a random $n_2$ such that $R_{n+1} \le R_n + \eta \ln R_n$ for
$n \ge n_2$. This implies estimates on the growth of $R_n$.
Indeed, fix $a > 1$ and define function $f: [1,\infty) \to \RR$ given by
$f(x) := x + \eta \ln x$. Notice that if we take $x$ of the form
$y (\ln y)^{a}$ we can write
\begin{align*}
f(x)
    &= y (\ln y)^{a} + \eta \ln [y (\ln y)^{a}]
    = y (\ln y)^{a} + \eta \ln y  + \eta a \cdot \ln \ln y \\
    &\le y (\ln y)^{a} + (\ln y)^{a}
    \le (y+1) (\ln (y+1))^{a}
\end{align*}
for any $y \ge y_0(a, \eta)$. Using that $R_n$ eventually grows at least
linearly, we can find a random $n_3 \ge n_2$ sufficiently large so that
$R_{n_3} = y (\ln y)^{a}$ for some $y \ge y_0$ and then
$R_{n+n_3} \le (y+n)(\ln (y + n))^{a}$ for every $n \ge 0$, which
implies that asymptotically we have
$\varlimsup \frac{R_n}{n (\ln n)^{a}} \le 1$. Since any choice of $a > 1$
works, the result follows.
\end{proof}

\medskip
\noindent
\textbf{Growth of heavy-tailed ERCP.}
The coupling between ERCP and iterated percolation we have just described
works for any sequence of times $(s_n)$ increasing to infinity and
satisfying~\eqref{eq:times_sn}. Recall the definition
\begin{equation*}
r_t := \max \{\norm{x}_1;\; (0,0) \rightsquigarrow (x,t)\
    \text{in ERCP without cures}\}.
\end{equation*}

\begin{proof}[Proof of Theorem~\ref{teo:growth_ercp}(i)]
Consider the sequence $s_n := n w_0(p_c(\ZZ^{d})/2)$.
We just have to combine the rate of growth of $s_n$ with the estimates given by
Proposition~\ref{prop:linear_Rn}. For any fixed time $t$, define
$n(t) := \lceil \frac{t}{w_0}\rceil$. Clearly, we have $r_t \le R_{n(t)}$ and
since $\varlimsup \frac{R_n}{n(\ln n)^{a}} = 0$ the result holds.
\end{proof}

Notice that in Theorem~\ref{teo:growth_ercp} the rate of growth of $s_n$
is essential in the final estimate. When $\mu$ is heavy-tailed, our estimate on
$r_t$ can be greatly improved by considering a sequence of times that grows
faster. We assume that $\mu$ satisfies~\eqref{eq:gap_t_epsilon}. Then, for
$n_0$ sufficiently large and the sequence of times $t_0 = 2^{n_0}$ and $t_{n+1}
= t_{n} + t_{n}^{\epsilon_4}$ we can ensure
\begin{equation}
\label{eq:subcritical_steps}
\PP(\cR \cap [t_n, t_{n+1}] \neq \emptyset) \le p < p_c(\ZZ^{d})
\end{equation}
and let $\cP_{n}$ be independent Bernoulli bond percolation models
with parameter $p$. In other words, we start the coupling only at time $t_0$,
when we have a (random) finite infected set $\cI_0$ and then
\begin{equation}
\label{eq:coupling}
\{x;\; (0,0) \rightsquigarrow (x, t_n)\ \text{in ERCP without cures}\}
    \subset \cI_n
    \subset \bar{\cC}_n(\cI_{n-1})
    \quad \text{for every $n \ge 1$}.
\end{equation}

\begin{proof}[Proof of Theorem~\ref{teo:growth_ercp}(ii)]
The growth of $R_n$ is estimated in Proposition~\ref{prop:linear_Rn}. Fixing
$a>1$, we notice that $r_{t_n} \le R_n \le n^{a}$ for $n$ large. It is also
clear that $r_t$ is non-decreasing. Thus, for some $s > 0$ sufficiently large,
if we define $n = n(s)$ as the unique integer satisfying $t_{n-1} < s \le
t_{n}$, then $r_s \le n^{a}$. We just have to estimate $n(s)$.

Consider intervals $I_i := [2^{n_0 + i - 1}, 2^{n_0 + i}]$. Each of them
cannot have too many points of sequence $(t_j)$. Indeed, since for
$t_j \in I_i$ we have $t_{j+1} - t_j = t_j^{\epsilon_4}
\ge 2^{(n_0+i-1) \epsilon_4}$, we have
\begin{equation*}
\# \{j;\; t_j \in I_i\}
    \le \frac{2^{n_0 + i - 1}}{2^{(n_0 + i - 1) \epsilon_4}}
    = 2^{(n_0 + i - 1) (1-\epsilon_4)}.
\end{equation*}
Since $s \ge 2^{\lfloor \log_2 s \rfloor}$, we conclude that
\begin{equation*}
    n(s) \le \sum_{i=1}^{\lfloor \log_2 s \rfloor - n_0} 2^{(n_0 + i - 1) (1-\epsilon_4)}
    \le c(n_0, \epsilon_4) 2^{\lfloor \log_2 s \rfloor (1-\epsilon_4)}
    \le c(n_0, \epsilon_4) s^{1-\epsilon_4}.
\end{equation*}
Notice that since we can take any $a>1$, the result
in~\eqref{eq:sublinear_ERCP_rcp1} follows.
\end{proof}

\begin{coro}
\label{coro:growth_ercp}
If $\mu(t, \infty) = L(t)t^{-\alpha}$ with $\alpha \in (0,1)$,
$L(t)$ slowly varying, and $\mu$ satisfies the Strong Renewal Theorem
(cf.~\cite{CD}), then for all $\eta > 0$ we have
\begin{equation}
    \label{eq:sublinear_ERCP_SRT}
    \varlimsup_{s \to \infty} \frac{r_s}{s^{\alpha + \eta}} \le 1.
\end{equation}
\end{coro}

\begin{proof}
The estimate in~\eqref{eq:sublinear_ERCP_SRT} follows in the same manner,
by noticing that if we have the Strong Renewal Theorem then
\begin{equation*}
\PP(\cR \cap [t,t+t^{\epsilon}] \neq \emptyset)
    \le \smash{\sum_{s=1}^{t^{\epsilon}}} \PP(\cR \cap [t+s-1,t+s] \neq \emptyset)
    \sim t^{\epsilon} c_{\alpha} \frac{L(t)}{t^{1-\alpha}}
\end{equation*}
for some slowly varying function $L$. Hence, this probability goes to zero
whenever $\epsilon < 1-\alpha$. This implies that we can take $\epsilon_4$
arbitrarily close to $1-\alpha$ and the result follows.
\end{proof}

\subsection{Extinction in heavy-tailed ERCP}
\label{sub:extinction_ercp}

Theorem~\ref{teo:growth_ercp}(ii) gives a bound on how fast the infection
can spread without any cures: for large $t$, the infection is contained inside
$\{(x,t); \norm{x}_1 \le t^{\rho}\}$ for some $\rho < 1$.

Fix $\beta > 0$. The next step in our investigation is to show that eventually
we are able to cure all the infection in a region of the form $\sB(2^{\beta n})
\times [2^{n}, 2^{n}+2^{n \epsilon_4}]$. Choosing $\beta$ sufficiently large,
this will imply that the process dies almost surely for any renewal process
$\nu_{\delta}$ for the cures, given that $\nu$ has moments of all orders.
The following lemmas introduce some bad events that would make it
more difficult for curing all the infection at once. We show that each of these
events cannot happen infinitely often.

Our first lemma estimates the probability of having large clusters of
transmissions inside $\sB(2^{\beta n})$. For that, we recall that a finite
connected subgraph of $\ZZ^{d}$ that contains the origin is said to be an
\textit{animal}. Let us denote by $A_m$ the set of animals with $m$ edges. By
Equation~(4.24) of~\cite{Gri} we have that $\# A_m \le 7^{d(m+1)} \le C^{m}$
for some positive constant $C(d)$.  The probability finding a cluster of $m$
adjacent transmissions in region $\sB(2^{\beta n}) \times [2^{n}, 2^{n}+2^{n
\epsilon_4}]$ decays quickly with $n$.

\begin{lema}
\label{lema:defi_Un}
Let $\mu$ satisfy~\eqref{eq:gap_t_epsilon}. Consider the event $U_n = U_n(m, \beta)$
defined by
\begin{equation}
\label{eq:animal_event}
    U_n
    := \bigcup_{x\in \sB(2^{\beta n})} \bigcup_{M \in x+A_m}
    \{\cR_e \cap [2^{n}, 2^{n}+2^{n \epsilon_4}] \neq \emptyset,\
    \text{for every $e$ edge of $M$} \}
\end{equation}
There is $m(\epsilon_4, \beta, d) \in \NN$ such that
$\PP\bigl(\varlimsup_n U_n \bigr) = 0$.
\end{lema}

\begin{proof}
Using the union bound and the estimate from~\eqref{eq:gap_t_epsilon},
we can write
\begin{equation*}
\PP(U_n)
    \le c(d)2^{d \beta n} \cdot C^{m} \cdot 2^{- n \epsilon_4 m}
    = cC^{m} \cdot 2^{(d\beta - \epsilon_4 m) n}
\end{equation*}
and it suffices to choose $m > \frac{d \beta }{\epsilon_4}$ to make
$\sum_n \PP(U_n)$ summable.
\end{proof}

A second estimate that is useful is a consequence of~\cite[Lemma~3]{FMMV}. It
says that even when there are transmissions in an interval $[2^{n}, 2^{n}+2^{n
\epsilon_4}]$ for an edge of $\sB(2^{\beta n})$, the probability of having too
many transmissions in this edge decays fast with $n$.
\begin{lema}
\label{lema:defi_Vn}
Let $\mu$ satisfy~C). Let $V_n(\epsilon_4, \beta, \eta)$ be the event
\begin{equation}
\label{eq:defi_Vn}
V_n
    := \{\text{$\exists e$ edge of $\sB(2^{\beta n})$};\;
        |\cR_e \cap [2^{n}, 2^{n}+2^{\epsilon_4 n}]| \ge 2^{n \epsilon_4 \eta}\}.
\end{equation}
There is $\eta = \eta(\mu)$ with $\eta \in (0,1)$ such that
it holds $\PP(\varlimsup_n V_n) = 0$.
\end{lema}

\begin{proof}
Taking $I = [2^{n}, 2^{n}+2^{n \epsilon_4}]$ and denoting by $l = 2^{n
    \epsilon_4}$
its length, Lemma~3 of~\cite{FMMV} shows that
\begin{equation*}
\PP(|\cR \cap I| \ge l^{1-\epsilon_3} \ln^{2} l)
    \le 2 \cdot e^{- \ln^{2} l}
    \le 2^{- c \epsilon_4^2 n^{2}}
    \qquad \text{for large $n$ and some $c>0$},
\end{equation*}
where constant $\epsilon_3 > 0$ satisfies
$\mu(t,\infty) \ge t^{-(1-\epsilon_3)}$ for large $t$
(the proof of Lemma~3 of~\cite{FMMV} only uses the lower
bound of condition C)\ ). Taking $\eta \in (1-\epsilon_3, 1)$ we have
that
\begin{equation*}
l^{1-\epsilon_3} \ln^{2} l
    = 2^{(1-\epsilon_3)\epsilon_4 n} \ln^{2} 2^{\epsilon_4 n}
    \ll 2^{\eta \epsilon_4 n}
    \ll 2^{\epsilon_4 n}
\end{equation*}
for large $n$. The union bound implies
\begin{equation*}
\PP(V_n)
    \le K(d) 2^{d \beta n}
        \PP(|\cR \cap I| \ge 2^{n \epsilon_4 \eta})
    \le K(d) 2^{d \beta n} 2^{- c \epsilon_4^{2} n^{2}},
\end{equation*}
for some constant $K(d)>0$. Then, $\sum_n \PP(V_n)$ is summable
and the result follows.
\end{proof}

The last event we consider is the only one related to cures. Notice that on
event $U_n^{\comp} \cap V_n^{\comp}$ the region
$\sB(2^{\beta n}) \times [2^{n}, 2^{n}+2^{n \epsilon_4}]$ only has clusters
of transmissions with at most $m$ edges and each of these edges do not have
many transmissions. Hence, for each cluster $M$ we can find an interval
$I_M \subset [2^{n}, 2^{n}+2^{\epsilon_4 n}]$ with length at least
$2^{\epsilon_4 n}/(m2^{\eta \epsilon_4 n}) = 2^{(1-\eta) \epsilon_4 n}/m$
satisfying that $\cR_e \cap I_M = \emptyset$ for every edge of $M$.

Let us denote by $\cH_x$ the renewal process with interarrival $\nu_{\delta}$
that is associated to site $x \in \ZZ^{d}$.

\begin{lema}
\label{lema:defi_Wn}
Let $\beta, \epsilon_4 > 0$ and $\eta \in (0,1)$. Consider the event
$W_n = W_n(\beta, \epsilon_4, \eta)$ defined by
\begin{equation}
\label{eq:defi_Wn}
W_n
    := \bigcup_{x\in \sB(2^{\beta n})} \bigcup_{I}
        \{\cH_x \cap I = \emptyset\},
\end{equation}
where the second union is over all intervals
$I \subset [2^{n}, 2^{n}+2^{n \epsilon_4}]$ of length
$|I| = 2^{(1-\eta) \epsilon_4 n}/m$. If $\nu$
has finite moments of all orders then it holds
$\PP(\varlimsup_n W_n) = 0$.
\end{lema}

\begin{proof}
Let $t_j = 2^{n} + j \cdot 2^{(1-\eta) \epsilon_4 n}/(2m)$ and notice that
intervals $I_j = [t_j, t_{j+1}]$ for
$0 \le j \le \lceil 2m \cdot 2^{\eta \epsilon_4 n} \rceil$
cover $[2^{n}, 2^{n}+2^{n \epsilon_4}]$. If we have an interval $I$ of length
$|I| = 2^{(1-\eta) \epsilon_4 n}/m$ that has no cure marks for every
$x \in \sB(2^{\beta n})$, then this interval must contain some $I_j$. Hence,
\begin{equation*}
\PP(W_n)
    \le K(d) 2^{d \beta n}
        \sum_{j=0}^{\lceil 2m \cdot 2^{\eta \epsilon_4 n} \rceil}
        \PP(\cH \cap I_j = \emptyset).
\end{equation*}
By~\eqref{eq:moment_condition} in Lemma~\ref{lema:unif_estimates},
we can translate moments of $\nu_{\delta}$ into estimates for
$\PP(\cH \cap I_j = \emptyset)$. Consider the function
\begin{equation*}
f(x) := x^{a}, \qquad
\text{with $a > \frac{d \beta + \eta \epsilon_4}{(1-\eta)\epsilon_4}$}.
\end{equation*}
Since $E_{\nu_{\delta}}[Xf(X)] = E_{\nu}[(X/\delta)^{1+a}] < \infty$ and every
$I_j$ has length $2^{(1-\eta) \epsilon_4 n}/(2m)$, it follows that
\begin{equation*}
\PP(W_n)
    \le c(d, m, a, \delta) \cdot
        2^{(d \beta + \eta \epsilon_4 - a (1-\eta) \epsilon_4)n}.
\end{equation*}
Our choice of $a$ makes the coeficient multiplying $n$ in the exponent
negative, and we conclude that $\sum_n \PP(W_n)$ converges.
\end{proof}

From the estimates above, we have
\begin{proof}[Proof of Theorem~\ref{teo:ercp_extinction}(ii)]
By Theorem~\ref{teo:growth_ercp}(ii), there is $\rho(\mu) < 1$ such that
$r_t \le t^{\rho}$ for every large $t$. Fix $\beta > \rho$ and consider the bad
events $U_n(m, \beta), V_n(\epsilon_4, \beta, \eta),
W_n(\epsilon_4, \beta, \eta)$ described in Lemmas~\ref{lema:defi_Un},
\ref{lema:defi_Vn} and \ref{lema:defi_Wn} with $m$ and $\eta$ chosen so that
all the bounds in these lemmas hold.

Since $\beta > \rho$, for large $n$ we have that
$\sB(2^{\beta n}) \times [2^{n}, 2^{n}+2^{\epsilon_4 n}]$ will contain all
infected sites of time interval $[2^{n}, 2^{n}+2^{\epsilon_4 n}]$. Moreover,
on $U_n^{\comp} \cap V_n^{\comp} \cap W_n^{\comp}$ every cluster $M$
of transmissions cannot have more than $m$ edges and must have an interval
$I_M$ of length at least $2^{(1-\eta) \epsilon_4 n}/m$ without transmissions.
Since we also have that $I_M$ must have a point of $\cH_x$ for every
$x \in \sB(2^{\beta n})$, the result follows.
\end{proof}

\subsection{Phase transition}
\label{sub:phase_transition}

Lemma~\ref{lema:unif_estimates} is in the core of the proof of
Theorem~\ref{teo:ercp_extinction}(ii) and (iii).

\begin{proof}[Proof of Theorem~\ref{teo:ercp_extinction}(ii)]
We make a straightforward comparison with planar oriented percolation.
It is sufficient to prove the statement for $d=2$. We can actually prove there
is a positive probability of survival in the quadrant $\ZZ^2_+ \times \RR_+$.
In $\ZZ^{2}_+$, consider the graph with oriented edges $z+(1,0)$ and $z+(0,1)$.
Given any $\varepsilon>0$, notice that by Lemma~\ref{lema:unif_estimates}
we can find $h$ such that
\begin{equation*}
    \PP(\cR^{\mu} \cap [t, t+h] = \emptyset) \le \varepsilon,
    \qquad \text{for any $t \ge 0$}.
\end{equation*}
For each site $x \in \ZZ^{2}_+$ we associate a point $g(x) \in \ZZ^{2}_+ \times
\RR_+$ defined by $g(x) := (x, h \norm{x}_1)$, the only point in the
intersection of the vertical line from $(x,0)$ with the plane through
$(1, 0, h)$, $(0,1,h)$ and the origin.
We denote by $G$ the graph with vertex set  $\{g(x);\; x \in \ZZ^{2}_+\}$
and oriented edges from $g(x)$ to $g(y)$ if and only if there is one
from $x$ to $y$ in the oriented graph $\ZZ^{2}_+$. Clearly, $G$ is
isomorphic to $\ZZ^{2}_+$ as shown in Figure~\ref{fig:embed_ZZ2+}.
Define a site bond percolation model in $G$ by stating that
\begin{itemize}
\item A site $g(x)$ is open if and only if
    $\{x\} \times [h(\norm{x}_1 - 1), h(\norm{x}_1 + 1)]$ has no cure marks.
\item An edge from $g(x)$ to $g(y)$ is open if and only if
    $\cR_{x,y} \cap [h(\norm{x}_1 - 1), h\norm{x}_1] \neq \emptyset$.
\end{itemize}

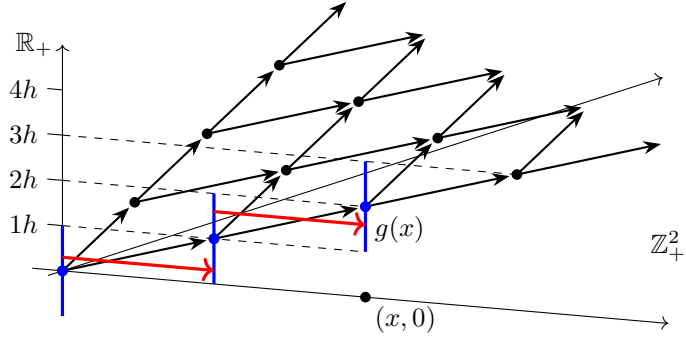
\begin{figure}[h]
\centering
\begin{tikzpicture}[scale=1,
    x={( 18:1cm)}, y={(-5:2cm)}, z={(0cm,.6cm)},
    dot/.style={fill, minimum size=4pt, outer sep=0pt,
                    inner sep=0pt, circle}
    ]
    \draw[->] (0,0,-.2) -- (0,0,5) node[left] {$\RR_+$};
    \draw[->] (0,-.2,0) -- (0,4,0) node[above=.7cm] {$\ZZ^{2}_+$};
    \draw[->] (-.2,0,0) -- (8.3,0,0);

\foreach \z in {1, ..., 4}{
    \draw (0,0,\z) -- ++(0,-.1) node[left] {$\z h$};
};

\foreach \z in {0, ..., 3}{
    \foreach \x in {0, ...,\z}{
        \draw[thick, -{Stealth}, shorten >=1mm] (\x, \z-\x, \z) -- ++(1,0,1);
        \draw[thick, -{Stealth}, shorten >=1mm] (\x, \z-\x, \z) -- ++(0,1,1);
        \node[dot] at (\x, \z-\x, \z) {};
    };
};

\node[dot, blue] at (0,0,0) {};
\node[dot, blue] at (0,1,1) {};
\node[dot] (x) at (0,2,0) {}
    node[below right] at (x) {$(x,0)$};
\node[dot, blue] (gx) at (0,2,2) {}
    node[below right] at (gx) {$g(x)$};
\draw[dashed] (0,0,1) -- (0,2,1) (0,0,2) -- (0,2,2) (0,0,3) -- (0,3,3);
\draw[very thick,blue] (0,0,-1) -- (0,0,1) (0,1,0) -- (0,1,2) (0,2,1) -- (0,2,3);
\draw[very thick, red, ->] (0,0,.3) -- (0,1,.3);
\draw[very thick, red, ->] (0,1,1.6) -- (0,2,1.6);

\end{tikzpicture}
\caption{Coupling with oriented site bond percolation when $d \ge 2$. A path of
    open sites and bonds connecting $0$ to $g(x)$ in $G$ implies the infection
    reaches $g(x)$ in the original model.}
\label{fig:embed_ZZ2+}
\end{figure}

Notice that the state of every site and bond is independent and the
probability of an edge being open is at least $1-\varepsilon$. Now, we use
Lemma~\ref{lema:unif_estimates}(ii) to obtain a similar estimate for the
probability of a site being open. There is $w_0(\varepsilon) > 0$ such that it
holds
\begin{equation*}
\inf_{t\ge 0} \PP( \cR^{\nu} \cap [t, t + w] = \emptyset)
    \ge 1 - \varepsilon,
\qquad \text{for any $w \in (0, w_0)$}.
\end{equation*}
Notice that $\PP( \cR^{\nu} \cap [t, t + w] = \emptyset)
    = \PP( \cR^{\nu_{\delta}} \cap
    [\frac{1}{\delta}t, \frac{1}{\delta}t + \frac{1}{\delta}w]
    = \emptyset)$ for every $t \ge 0$.
Taking $\delta_0 := \frac{w_0}{2h}$, it follows that for $\delta \in (0,
\delta_0)$ the probability of a site being open is at least $1-\varepsilon$.
This independent site-bond model can be compared with a finite range dependent
bond model: say that an edge $e = (x,y)$ is open if $e, x$ and $y$ are all
open in $G$ leads to a bond model in which edges that do not share extremities
are independent. By the classical stochastic domination results of Liggett,
Schonmann and Stacey~\cite{LSS}, if we choose $\varepsilon>0$ small enough
in the beginning it follows that ERCP($\mu, \nu_{\delta}$) with
$\delta \in (0, \delta_0)$ survives with positive probability.
\end{proof}

\begin{proof}[Proof of Theorem~\ref{teo:ercp_extinction}(iii)]
Here we use the recurrence inequality approach. The initial part of the
argument is essentially the same as in the proof of~\cite[Theorem~1.1]{FMUV},
with the only difference that now transmissions and cures are given
by renewal processes with interarrival distributions $\mu$ and $\nu_{\delta}$,
respectively. Analogous to~\cite[Definition~2.2]{FMUV}, consider the
uniform quantities
\begin{equation}
\label{eq:uniform_quantities}
\tilde{s}_n
    := \sup \hP(\tS_j((x,t) + B_n))
\quad \text{and} \quad
\tilde{t}_n
    := \sup \hP(\tT((x,t) + B_n)),
\end{equation}
where the suprema above are over all $(x,t) \in \ZZ^{d} \times \RR_+$ and all
product renewal probability measures $\hP$ with interarrival distributions
$\mu$ and $\nu_{\delta}$ and renewal points starting at (possibly different)
time points strictly less than zero. Define
\begin{equation}
\label{eq:defi_un_ercp}
    u_n := \tilde{s}_n + \tilde{t}_n.
\end{equation}

Most of the reasoning in the proof of~\cite[Theorem~1.1]{FMUV} still holds,
up to the choice of box sequence. More precisely, consider boxes
$B_n = [0, 2^{n}]^{d} \times [0, h_n]$. Using~\eqref{eq:moment_condition} in
Lemma~\ref{lema:unif_estimates}(i) we estimate the probability of decoupling
both transmissions and cures. For some fixed increasing function $f$ to be
precised later, we estimate the probability that some site or edge in
$[0,2^{n}]^{d}$ does not have a mark in the interval $[t, t+h]$ by
\begin{equation*}
\sup_{t \ge 0}
    \hP\Bigl(
        \bigcup_e \{\cR^{\mu}_{e} \cap [t, t+h] = \emptyset\} \cup
        \bigcup_x \{\cR^{\nu_{\delta}}_{x} \cap [t, t+h] = \emptyset\}
        \Bigr)
    \le \frac{d2^{dn}C(f, \mu)}{f(h)} + \frac{2^{dn}C(f,\nu_{\delta})}{f(h)}.
\end{equation*}
More than that, the upper bound above can be taken uniform for $\delta > 1$
since
\begin{equation}
\label{eq:decoup_nu}
\PP(\cR^{\nu_{\delta}}_{x} \cap [t, t+h] = \emptyset)
    = \PP(\cR^{\nu}_{x} \cap [\delta t, \delta t+ \delta h] = \emptyset)
    \le \frac{C(f, \nu)}{f(\delta h)}
    \le \frac{C(f, \nu)}{f(h)}.
\end{equation}
Hence, the same line of reasoning in the proof of~\cite[Theorem~1.1]{FMUV}
shows there are constants $c(d)$ and $C(d, \mu, \nu, f)$ such that
\begin{equation}
\label{eq:recurrence_un}
u_{n}
    \le c \cdot (h_n/h_{n-1})^{2} \cdot u_{n-1}^{2} +
        \frac{C 2^{dn}}{f(h_{n-1})}.
\end{equation}
Since $\mu$ and $\nu$ satisfy~\eqref{eq:moment_condition_rcp3}, if we choose
sequence $h_n = e^{(\alpha/\theta)^{2}n^{2}}$ with an appropriate choice of
$\alpha$ like in the proof of~\cite[Lemma~2.7]{FMUV}, it follows that there
is $n_0(\mu, \nu, \theta, d)$ such that if $u_{n_0} \le 2^{-d n_0}$ then $u_n
\le 2^{-dn}$ for every $n \ge n_0$. To finish the proof, we only need to
choose $\delta > 1$ sufficiently high so that $u_{n_0} \le 2^{-d n_0}$.

Recall that $n_0(d, \mu, \nu, \theta)$ is fixed, and so are the dimensions
$l_{n_0}$ and $h_{n_0}$ of box $B_{n_0}$. To control the
probability of $\hP(\tS_1((x,s) + B_{n_0}))$ and $\hP(\tT((x,s) + B_{n_0}))$
uniformly in $(x,s)$ and $\hP$ we observe the following.
Firstly, we control the probability of some edge in $\pi(B_{n_0})$ having
too many renewal marks. Let $N$ denote the number of edges in
$\pi(B_{n_{0}})$. For a single edge, we can find $k_0(n_0, \mu)$
sufficiently large so that
\begin{equation*}
\PP(\# \cR_e \cap [s,s+h_{n_0}] \ge k_0) \le \frac{1}{4N} 2^{- d n_0},
\end{equation*}
implying that
\begin{equation}
\label{eq:many_transmissions}
\PP\Bigl(
    \bigcup_{e \in \pi(B_{n_0})} \{\# \cR_e \cap [s,s+h_{n_0}] \ge k_0 \}
\Bigr)
    \le N \cdot \PP(\# \cR \cap [s,s+h_{n_0}] \ge k_0)
    \le \frac{1}{4} 2^{- dn_{0}}.
\end{equation}
Estimate~\eqref{eq:many_transmissions} controls the probability that some
edge has more than $k_0$ renewals. When every edge of $\pi(B_{n_0})$ has less
than $k_0$ renewals, we show that after every transmission in $B_{n_0}$ there
is a high probability that every site gets cured before the next transmission
if $\delta$ is sufficiently large.

Let $Z_t$ denote the overshoot at time $t$, i.e.,
$Z_t := \inf \cR \cap [t, \infty) - t$. We denote by $Z^{\mu, e}_t$ and
$Z^{\nu_{\delta}, x}_t$ the overshoots of the renewal processes of edge $e$
and site $x$. By Lemma~\ref{lema:unif_estimates}(ii), given
$\varepsilon > 0$ there is $w_0(\varepsilon, \mu) > 0$ such that
\begin{equation}
\label{eq:min_Zet}
\inf_{t \ge 0} \hP(\min_{e \in \pi(B_{n_0})} Z_t^{\mu,e} > w_0)
    \ge (1-\varepsilon)^{N}.
\end{equation}
On the other hand, denoting by $M$ the number of sites in $\pi(B_{n_0})$
we have that
\begin{align*}
\inf_{t \ge 0} \hP(\max_{x \in \pi(B_{n_0})} Z_t^{\nu_{\delta},x} \le w_0)
    &= \bigl(1 - \sup_{t \ge 0} \hP(Z_t^{\nu_{\delta},x} > w_0)\bigr)^{M}
    \ge \bigl(1 - \sup_{t \ge 0}
        \PP(\cR^{\nu_{\delta}} \cap [t, t+w_0] = \emptyset)
        \bigr)^{M}\\
    &\ge \Bigl(1 - \frac{C(\nu, f)}{f(\delta w_0)}\Bigr)^{M},
\end{align*}
where the last inequality is a consequence of~\eqref{eq:decoup_nu}.
Hence, we can choose $\delta_0(\nu, w_0, f, \varepsilon) > 0$ so that for
$\delta \ge \delta_0$ we have
\begin{equation}
\label{eq:max_Zxt}
\inf_{t \ge 0} \hP(\max_{x \in \pi(B_{n_0})} Z_t^{\nu_{\delta},x} \le w_0)
    \ge (1-\varepsilon)^{M}.
\end{equation}
Combining~\eqref{eq:min_Zet} and~\eqref{eq:max_Zxt} we can write that
\begin{equation}
\label{eq:max_Zxt_before_min_Zet}
\inf_{t \ge 0} \hP(\min_{e} Z_t^{\mu,e} > \max_{x} Z_t^{\nu_{\delta},x})
    \ge \inf_{t \ge 0}
        \hP(\min_{e} Z_t^{\mu,e} > w_0 \ge \max_{x} Z_t^{\nu_{\delta},x})
    \ge (1-\varepsilon)^{N+M}.
\end{equation}

Let us estimate $\hP(\tS_1((x,s) + B_{n_0}))$. Define $T_0 := s$ and
$T_{j+1} := T_j + \min_{e \in \pi(B_{n_0})} Z_{T_j}^{e}$. Then, $T_j$ is an
increasing sequence of stopping times for the filtration
$\cF_t$ that makes $\cR_e \cap [0,t]$ and $\cR_x \cap [0,t]$ measurable for
every edge $e$ and site $x$ in $\pi(B_{n_0})$.
By~\eqref{eq:many_transmissions}, we know that
\begin{equation*}
\hP(T_{Nk_0} \le s + h_{n_0})
    \le \frac{1}{4} 2^{- dn_{0}}.
\end{equation*}

Given $\cF_{T_j}$, notice that at $T_j$ there has been unique transmission in
$(x,s) + B_{n_0}$ and by~\eqref{eq:max_Zxt_before_min_Zet} all sites will cure
before the next transmission with probability
\begin{equation*}
\hP(
    \min_{e} Z_{T_{j}}^{\mu,e} > \max_{x} Z_{T_{j}}^{\nu_{\delta},x}
    \mid \cF_{T_j})
    \ge (1-\varepsilon)^{N+M},
\end{equation*}
implying that
\begin{equation*}
\hP\Bigl( \bigcap_{j=1}^{Nk_0}
    \{\min_{e} Z_{T_{j}}^{\mu,e} > \max_{x} Z_{T_{j}}^{\nu_{\delta},x}\}
    \Bigr)
    \ge (1-\varepsilon)^{N(N+M)k_0}.
\end{equation*}
Notice that on the event
\begin{equation*}
\{T_{Nk_0} > s + h_{n_0}\} \cap
\Bigl(\bigcap_{j=1}^{Nk_0}
    \{\min_{e} Z_{T_{j+1}}^{\mu,e} > \max_{x} Z_{T_{j+1}}^{\nu_{\delta},x}\}
\Bigr)
\end{equation*}
there are no spatial nor temporal crossings. Hence, it follows
\begin{equation}
\label{eq:unif_tS_estimate}
\hP(\tS_1((x,s) + B_{n_0}))
    \le \frac{1}{4} 2^{- dn_{0}} + \bigl[1 - (1-\varepsilon)^{N(N+M)k_0}\bigr]
    \le \frac{1}{2} 2^{- dn_{0}}
\end{equation}
if we choose $\varepsilon(n_0, k_0, d)>0$ sufficiently small. We emphasize that
the choice of parameters above is not circular: we can choose in order
$k_0(n_0, \mu)$, $\varepsilon(n_0, k_0, d)$, $w_0(\varepsilon, \mu)$ and then
$\delta_0(\nu, w_0, f, \varepsilon)$, so that~\eqref{eq:unif_tS_estimate} holds
for $\delta \ge \delta_0$, uniformly on $\hP$ and $(x,s)$. We can estimate
$\hP(\tT((x,s) + B_{n_0}))$ in a similar way, which leads to
$u_{n_0} \le 2^{- d n_0}$ and the result follows.
\end{proof}

\end{document}